\documentclass[12pt]{amsart}
\usepackage{amsmath,amssymb,latexsym,amsmath,amsthm,amscd}
\usepackage{setspace}

\theoremstyle{definition}
\theoremstyle{plain}

\usepackage[left=2.8cm,top=2.8cm,right=2.8cm,bottom = 2.4cm]{geometry}
\usepackage{amsfonts}
\usepackage{amsmath}
\usepackage{amscd}
\usepackage{amssymb}
\usepackage{amsthm}
\usepackage{bbm}
\usepackage{bm}
\usepackage{enumerate}
\usepackage{enumitem}
\usepackage{hyperref}
\usepackage{latexsym}
\usepackage{mathabx}
\usepackage{mathdots}
\usepackage{mathtools}

\allowdisplaybreaks

\DeclareMathOperator{\coeff}{coeff}
\DeclareMathOperator{\diag}{diag}
\DeclareMathOperator{\dom}{dom}
\DeclareMathOperator{\ima}{im}

\DeclareMathOperator{\perm}{perm}
\DeclareMathOperator{\pperm}{pperm}
\DeclareMathOperator{\sgn}{sgn}
\DeclareMathOperator{\SL}{SL}
\DeclareMathOperator{\Tr}{Tr}

\newcommand{\Ah}{\widehat{A}}
\newcommand{\bbar}{\overline{b} }
\newcommand{\Bh}{\widehat{B}}
\newcommand{\bi} {\bm{i}}

\newcommand{\br}{\bm{r}}

\newcommand{\brfact} {\bm{r(i_{a})!}}

\newcommand{\C}{\mathbb{C}}
\newcommand{\calY}{\mathcal{Y}}
\newcommand{\calS}{\mathcal{S}}
\newcommand{\Cg}{\mathfrak{C}_{g}}
\newcommand{\Chat}{\widehat\C}
\newcommand{\Con}[1]{\mathcal{C}_{#1}}
\newcommand{\Consig}[1]{\mathcal{C}^{\sigma}_{#1}}
\newcommand{\D}{\mathcal{D}}

\newcommand{\F}{\mathcal{F}}
\newcommand{\half}{\frac{1}{2}}
\newcommand{\I}{\mathcal{I}}
\newcommand{\Ip}{\I_{+}}
\newcommand{\im}{\textup{i}}
\newcommand{\M}{\mathcal{M}}

\newcommand{\om}[2]{{\omega}_{{#2}}^{(#1)}}
\newcommand{\Omo}{\Omega_{0}(\Gamma)}
\newcommand{\R}{\mathbb{R}}
\newcommand{\Schg}{\mathfrak{S}_{g}}
\newcommand{\Sg}{\mathcal{S}_{g}}

\newcommand{\tpi}{2\pi\im}
\newcommand{\Uh}{\widehat{U}}
\newcommand{\Vh}{\widehat{V}}
\newcommand{\vac}{\mathbbm{1}}
\newcommand{\wt}{\textup{wt}}
\newcommand{\Z}{\mathbb{Z}}

\newtheorem{theorem}{Theorem}[section]
\newtheorem{proposition}[theorem]{Proposition}
\newtheorem{corollary}[theorem]{Corollary}
\newtheorem{lemma}[theorem]{Lemma}

\newtheorem{remark}[theorem]{Remark}
\newtheorem{example}[theorem]{Example}

\begin{document}
\title{The Heisenberg Generalized Vertex Operator Algebra on a Riemann Surface}
\author{
 Michael P. Tuite}
\address{School of Mathematics,  
Statistics and Applied Mathematics, 
National University of Ireland Galway,
University Road, Galway, Ireland.
}
\maketitle

\begin{abstract} 
We compute the partition and correlation generating functions for the Heisenberg
intertwiner generalized vertex operator algebra on a genus $g$ Riemann surface in
the Schottky uniformization. These are expressed in terms of differential forms of the
first, second and third kind, the prime form and the period matrix and are computed by combinatorial methods using a generalization of the MacMahon Master Theorem.
\end{abstract}

\section{Introduction} 
The Heisenberg intertwiner generalized Vertex Operator Algebra (VOA) is an algebra
formed from the Heisenberg VOA and  all of its modules \cite{DL, BK, TZ}. 
We consider the partition and all correlation functions for this theory on a genus $g$ Riemann surface in the Schottky parametrization. 
The partition and $n$-point correlation functions for the Heisenberg and lattice VOAs are familiar concepts at genus one and have been found on genus two surfaces formed from sewn tori \cite{MT1, MT2}. 
Here we describe the more general situation of the Heisenberg generalized VOA and compute the partition function and the generating function for all correlation functions on a genus $g$ Riemann surface.
Our results imply that we can compute all genus $g$ correlation functions for all VOAs or Super VOAs which can be decomposed into Heisenberg modules at any rank e.g. integral lattice (Super)VOAs. 
Various specializations of our results have been long anticipated in physics e.g. \cite{Mo, DV}. 
We also show that the rank 2 Heisenberg VOA genus $g$ partition function is an inverse determinant given by the Motonen-Zograf formula \cite{Mo, Z, McIT}.
All of our results are found by combinatorial methods based on a generalization of the MacMahon Master Theorem (MMT) \cite{McM, T} which is reviewed in Section~2. Section~3 reviews some Riemann surface theory and the Schottky uniformization of a genus $g$ surface in a non-standard parameterization suitable for our purposes. We also give detailed formulas for the bidifferential form of the second kind, holomorphic 1-forms (differentials of the first kind) and the period matrix in terms of genus zero data and Schottky sewing parameters \cite{Y, MT1}. 
Section~4 describes the genus $g$ partition and correlation generating function for the Heisenberg VOA. For convenience we consider the rank 2 case wherein the genus zero correlation generating function is expressed as a permanent. 
The genus $g$ objects can then be expressed as sums over multisets of permanents where the multisets label Heisenberg Fock vectors. 
The MMT implies that the genus $g$ partition function is the inverse of an infinite determinant (similarly to genus two \cite{MT4, MT5}) related to the Motonen-Zograf formula \cite{Mo, Z, McIT}. 
The Heisenberg VOA correlation generating function is expressed in terms of bidifferential forms. 
Section~5 generalizes all of these results to the case of Heisenberg generalized VOA where the correlation generating function is expressed in terms of differential forms of the first, second and third kind, the prime form and the period matrix. 
We conclude with a few important examples.
\section{Some Generalized MacMahon Master Theorems}
We begin with a review of some generalisations of the classic MacMahon Master Theorem  (MMT) of combinatorics  \cite{McM}  described in our earlier paper \cite{T}. Let $A=(A_{ab})$ be an $n\times n$  formal matrix indexed by $a,b\in\{1,\ldots ,n\}$. The Permanent of $A$ is defined by
\begin{align*}
\perm A:=\sum\limits_{\pi \in  S_{n}} \prod\limits_{i=1}^{n}A_{i\pi (i)},
\end{align*}
where the sum is over permutations $\pi\in  S_{n}$, the symmetric group on $n$ letters. 
Let 
\begin{align*}
\bi:=\{1^{r(1)}2^{r(2)}\ldots i^{r(i)}\ldots n^{r(n)}\},
\end{align*}
denote the multiset of size $N=\sum_{i=1}^{n}r(i)$ formed from the original index set 
$\{1,\ldots ,n\}$ where the index $i$ is repeated $r(i)\ge 0$ times. 
We let $A(\bi,\bi)$ denote the $N\times N$ matrix indexed by the multiset $\bi$ and 
define $\perm A(\bi,\bi)=1$ when $\bi$ is the empty set. Lastly,  we let 
$\br(\bi)!:=\prod_{i=1}^{n} r(i)!$ which is the order of the symmetric label group of $\bi$. Then \cite{McM}
\begin{theorem}[MMT]
	\label{theor:MMT}
\begin{align}
\sum_{\bi}\frac{\perm A(\bi,\bi)  }{\br(\bi)!}\, 
=
\frac{1}{\det (I-A)} ,
\label{eq:MMT}
\end{align}
where the sum is taken over all multisets $\bi$.
\end{theorem}
We review several generalizations of this result \cite{T}. 
Consider an $(n^{\prime }+n)\times (n^{\prime }+n)$
matrix with block structure 
\begin{align}
\begin{bmatrix}
B & U \\ 
V & A
\end{bmatrix},
\label{eq:Block}
\end{align}%
where $A=(A_{ij})$ is an $n\times n$ matrix indexed by $i,j$, $B=(B_{i^{\prime }j^{\prime }})$ is an $n^{\prime }\times n^{\prime }$
matrix indexed by $i^{\prime },j^{\prime }$, $U=(U_{i^{\prime }j})$ is an $n^{\prime }\times n$ matrix
and $V=(V_{ij^{\prime }})$ is an $n\times n^{\prime }$ matrix. For a
multiset $\bi$ of size $N$ define 
the $(n^{\prime }+N)\times (n^{\prime }+N)$ matrix 
\begin{align}
\begin{bmatrix}
B & U(\bi) \\ 
V(\bi) & A(\bi,\bi)
\end{bmatrix},
\label{eq:BAk}
\end{align}%
where, as before, $A(\bi,\bi)$ denotes the $N\times N$ matrix indexed by 
$\bi$, $U(\bi)$ is an $n^{\prime }\times N$ matrix and $V(\bi)$ is an 
$N\times n^{\prime }$ matrix. 
We  find\footnote{In  \cite{T}, a further parameter $\beta$ which counts permutation cycles is also discussed. We take $\beta=1$ throughout the present paper.}  \cite{T}
\begin{theorem}[The Submatrix MMT]
	\label{theor:MMTsub} 
	\begin{align}
	\sum_{\bi}\frac{1}{\br(\bi)!}\,\perm
	\begin{bmatrix}
	B & U(\bi) \\ 
	V(\bi) & A(\bi,\bi)
	\end{bmatrix}
	= \frac{\perm\widetilde{B}}{\det (I-A)},
	\notag
	\end{align}%
	for $n'\times n'$ matrix  $
	\widetilde{B}=B+U(I-A)^{-1}V$ where $(I-A)^{-1}=\sum_{k\ge 0}A^{k}$.
\end{theorem}
We may extend the  sum over permutations in \eqref{eq:MMT} to a sum over partial permutations i.e.  injective partial mappings from $\{1,\ldots ,n\}$ to itself.  
Let $\Psi$ denote the set of partial permutations of the set  $\{1,\ldots ,n\}$ and let $\dom \psi$ and $\ima \psi$ denote the domain and image respectively of $\psi\in\Psi$. 
Let $\theta=(\theta_{i}),\phi=(\phi_{i})$ be formal $n$-vectors. 
We define the $(\theta,\phi)$-extended Partial Permanent of an $n\times n$ matrix $A$ by  \cite{T}
\begin{align}
\pperm_{\,\theta,\phi} A:= \sum\limits_{\psi\in\Psi }
\prod\limits_{i \in \dom \psi} A_{i\psi(i)}
\prod\limits_{j \not\in \, \ima \psi} \theta_{j}
\prod\limits_{k \not\in \, \dom \psi} \phi_{k}
,
\label{eq:pperm}
\end{align}
Thus 
\begin{align*}
\pperm_{\,\theta,\phi}  
\begin{bmatrix}
A_{11} & A_{12}\\
A_{21} & A_{22}	
\end{bmatrix}
=&\theta_{1}\phi_{1}\theta_{2}\phi_{2}+A_{11}\theta_{2}\phi_{2}+A_{22}\theta_{1}\phi_{1}
\\
&+A_{12}\theta_{1}\phi_{2}+A_{21}\theta_{2}\phi_{1}+A_{11}A_{22}+A_{12}A_{21}.
\end{align*} 
Let $A(\bi,\bi)$ denote the $N\times N$ matrix indexed by $\bi$. 
We also let  $\pperm_{\,\theta,\phi}A(\bi,\bi)$ denote the corresponding partial permanent with dimension $N$ row vectors  $(\ldots,\theta_{i},\ldots)$ and $(\ldots,\phi_{i},\ldots )$ where index $i$ occurs $r(i)$ times in $\bi$. We then find  \cite{T}
\begin{theorem}[The Partial Permutation MMT]
	\label{theor:MMTpperm} 
	\begin{align*}
\sum_{\bi}\frac{\pperm_{\,\theta\phi } A(\bi,\bi)  }{\br(\bi)!}\, 
	=		\frac{e^{\theta(I-A)^{-1}\phi^{T}}}{\det (I-A)},
	\end{align*}%
	where $\phi^{T}$ denotes the transpose of the row vector $\phi$.
\end{theorem}
Lastly, we may combine the two generalizations above into one theorem concerning partial permutations of submatrices of the $(n'+n)\times (n'+n)$ block matrix  \eqref{eq:Block}. 
Let $\theta'=(\theta'_{i'})$ and $\phi'=(\phi'_{i'})$ 
be $n'$-vectors and $\theta=(\theta_{i})$ and $\phi=(\phi_{i})$ be $n$-vectors. 
For a multiset $\bi$ of size $N$ and block matrix \eqref{eq:BAk} labelled by 
$\{1',\ldots,n'\}$ and $\bi$, we let 
$\pperm_{\,\Theta,\Phi}
\begin{bmatrix}
B & U(\bi) \\ 
V(\bi) & A(\bi,\bi)
\end{bmatrix} $ 
denote the $(\Theta,\Phi)$-extended partial permanent with $(n'+N)$-vectors 
$\Theta:=(\theta'_{1'},\ldots,\theta'_{n'},\ldots,\theta_{i},\ldots)$ and 
$\Phi:=(\phi'_{1'},\ldots,\phi'_{n'},\ldots,\phi_{i},\ldots)$ respectively. 
We find  \cite{T}
\begin{theorem}[The Submatrix Partial Permutation MMT]
	\label{theor:MMTsubpperm} 
	\begin{align}
	\sum_{\bi}\frac{1}{\br(\bi)!}\, \pperm_{\,\Theta,\Phi}
	\begin{bmatrix}
	B & U(\bi) \\ 
	V(\bi) & A(\bi,\bi)
	\end{bmatrix}  
	=
			\frac{e^{\theta(I-A)^{-1}\phi^{T}}}{\det (I-A)}	\pperm_{\,\widetilde{\theta},\widetilde{\phi}}\widetilde{B},
	\label{eq:GenMMT}
	\end{align}%
 for $n'\times n'$ matrix  $\widetilde{B}= B+U(I-A)^{-1}V$ and $n'$-vectors $\widetilde{\theta} $ and $\widetilde{\phi}$ given by
	\begin{align*}
	\widetilde{\theta} &= \theta'+\theta(I-A)^{-1}V,\quad 
	\widetilde{\phi}^{T} = \phi'^{T}+U(I-A)^{-1}\phi^{T}.
	\end{align*}
\end{theorem}

\section{Riemann Surfaces from a Sewn Sphere }
\subsection{Some standard forms on a Riemann surface}
Define the  indexing sets
\begin{align*}
\I=\{-1,\ldots ,- g,1,\ldots ,g\},\quad
\Ip=\{1,\ldots ,g\}.
\end{align*}
Let $\Sg$  be a compact genus $g$ Riemann surface with canonical
homology basis $\alpha_{a}, \beta_{a}$ for $a\in \Ip$. 
There exists   a unique symmetric bidifferential form  of the second kind 
\cite{Mu,F}
\begin{align}
\omega(x,y)=\left(\frac{1}{(x-y)^{2}}+\text{regular terms}\right)\, dxdy,
\label{eq:omegag}
\end{align}%
for local coordinates $x,y$ with  normalization
$
\oint_{\alpha_{a}}\omega(x,\cdot )=0$ for all $a\in \Ip$.
It follows that
\begin{align}
\nu_{a}(x)=\oint\limits_{\beta_{a}}\omega(x,\cdot ),  \quad (a\in \Ip),
\label{eq:nu}
\end{align}%
is a differential of the first kind, a holomorphic 1-form normalized by
$\oint\limits_{\alpha_{a}}\nu_{b}=\tpi\, \delta_{ab}$. 
The  period
matrix $\Omega$ is defined by 
\begin{align}
\Omega_{ab}=\frac{1}{\tpi}\oint\limits_{\beta_{a}}\nu_{b}\quad \quad (a,b\in \Ip).
\label{eq:period}
\end{align}%
We also define the differential of the third kind for $p,q\in \Sg$ by
\begin{align}
\omega_{p-q}(x)=\int_{q}^{p}\omega(x,\cdot ).
\label{eq:omp2p1}
\end{align}
$\omega(x,y)$ can  be
expressed in terms of the prime form $E(x,y)=
K(x,y)dx^{-1/2}dy^{-1/2}$, a holomorphic form of weight
$(-\frac{1}{2},-\frac{1}{2})$ with  
\begin{align}
\omega (x,y) = &\partial_{x}\partial_{y}\left(\log K(x,y)\right)dxdy,
\label{eq:prime} 
\end{align}%
where $K(x,y)=(x-y)+O\left((x-y)^{2}\right)$ and  $K(x,y)=-K(y,x)$.
For the genus zero Riemann sphere $\calS_{0}\cong\Chat:=\C\cup\{\infty\}$ with $x,y,p,q\in \widehat\C$ we have
\begin{align}
\om{0}{}(x,y)= &\frac{dxdy,}{(x-y)^2}, \quad
\om{0}{p-q}=  \frac{dx}{x-p}-\frac{dx}{x-q} , \quad
K^{(0)}(x,y)= x-y. \label{eq:omK0}
\end{align}

\subsection{The Schottky uniformization of a Riemann surface}
We briefly review the  construction of a genus $g$ Riemann surface $\Sg$ using the Schottky uniformization where we sew $g$ handles to the Riemann sphere $\Chat$ e.g.  \cite{Fo,Bo}. 
For each $a\in\I$, let $\Con{a}\in \C$ be a circular contour with center $w_{a}$ and radius $|\rho_{a}|^{1/2}$ for some complex parameters $w_{a},\rho_{a}$. We assume that  $\rho_{a}=\rho_{-a}$ and that  $\Con{a}\cap \Con{b}=\{ \;\}$ for $a\neq b$.
Identify $z'\in \Con{-a}$ with $z\in  \Con{a}$ for each $a\in\Ip$ via the sewing relation  \cite{TW}
\begin{align}
\label{eq:sew}
(z'-w_{-a})(z-w_{a})=\rho_{a},\quad a\in\Ip.
\end{align}
Define $\gamma_{a}\in \SL_{2}(\C)$ by the  M\"obius map 
\begin{align}
\label{eq:gamaz}
\gamma_{a}z: = w_{-a}+\frac{\rho_{a}}{z-w_{a}}, \quad a\in\I.
\end{align} 
Notice that $\gamma_{-a}=\gamma_{a}^{-1}$. 
The sewing relation \eqref{eq:sew}  implies $z'=\gamma_{a}z$  for $a\in\Ip$  so that  $\gamma_{a} \Con{a}= -\Con{-a}$ for all $a\in\I$. 
\eqref{eq:sew} is equivalent to the standard Schottky  relation  \cite{Fo, Bo}
\begin{align}
\frac{z'-W_{-a}}{z'-W_{a}}\,\frac{z-W_{a}}{z-W_{-a}}=q_{a},\quad a\in\Ip,
\label{eq:Schottkysew}
\end{align}
where
$q_a=c\left(-\frac{\rho_a}{(w_a-w_{-a})^2}\right)$ 
for\footnote{$c(x)=\sum_{n\ge 1}\frac{1}{n} \binom{2n}{n+1}x^{n}$ is the Catalan series  \cite{MT1}.} 
$	c(x)=\frac{1-\sqrt{1-4x}}{2x}-1$ and $W_{a}= \frac{w_{a}+q_a w_{-a}}{1+q_a}$.  
Each $\gamma\in\Gamma$ is conjugate in $\SL_{2}(\C)$ to $\diag(q_{\gamma}^{1/2},q_{\gamma}^{-1/2})$ with   $|q_{\gamma}|<1$ where $q_{\gamma}$ is called the multiplier of $\gamma$.
In particular, $q_{\gamma_{a}}=q_{a}$ with attracting (repelling) fixed point $W_{-a}$  ($W_{a}$) for $a\in\Ip$. 

The (marked) Schottky group $\Gamma $  is the free group generated by $\gamma_{a}$. 
Every element of $\Gamma$ may be expressed as a reduced word $\gamma_{a_{1}}\ldots \gamma_{a_{n}}$ of length $n$ where $\gamma_{-a_{i}}\neq \gamma_{a_{i+1}}$. 
Let $\Lambda(\Gamma)$ denote the  limit set and let $\Omo=\Chat-\Lambda(\Gamma)$.
 Then  $\Sg\simeq \Omo/\Gamma$ is a Riemann surface of genus $g$.
We let $\D\subset\Chat$  denote the standard connected fundamental region with oriented boundary curves $\Con{a}$. 
We define the space of Schottky parameters  $\Cg \subset \C^{3g}$ by
\begin{align}
\label{eq:Cfrakg}
\Cg:=\left\{ (\bm{w,\rho}): |w_{a}-w_{b}|>|\rho_{a}|^{\frac{1}{2}}+|\rho_{b}|^{\frac{1}{2}}\; \forall \;a\neq b\right\},
\end{align}
for $\bm{w,\rho}:=w_{1},w_{-1},\rho_{1},\ldots,w_{g},w_{-g},\rho_{g}$. 
We define Schottky space as $\Schg:=\Cg/\SL_{2}(\C)$ for the M\"obius $\SL_{2}(\C)$ group. 

\subsection{Some Schottky scheme sewing formulas}
In this section we  generalize a construction due to Yamada  \cite{Y} and developed further in  \cite{MT1}
for sewing Riemann surfaces. 
In particular, we describe formulas for the genus $g$  normalized bidifferential of the
second kind  $\omega$, holomorphic 1-forms $\nu_{a}$
and the period matrix $\Omega_{ab}$ for $a,b\in\Ip$ constructed from  $\om{0}{}(x,y)$ and $\om{0}{p-q}$ of  \eqref{eq:omK0}  in the above Schottky sewing formalism. 
We note that there are classical formulas for these objects in terms of Poincar\'e sums that can be derived from the sewing formulas e.g. see \eqref{eq:omgSch}. However, the sewing expansions described here are much more suitable for our later purposes.
\begin{lemma}
\label{lem:inteqn2} 
\begin{align}
\omega (x,y)=\om{0}{}(x,y)
+\frac{1}{\tpi}\sum_{a\in\I}\,
\oint\limits_{\mathcal{C}_{a}(z_{a})} \int\limits^{z_{a}}\om{0}{}(x,\cdot )\;\omega (y,z_{a}),  \label{eq:w2rhointeqn}
\end{align}%
for $x,y\in \D$ with $z_{a}=z-w_{a}$. 
\end{lemma}

\begin{proof}
The result follows  from \eqref{eq:omegag}  and the identity  
\[
\oint\limits_{\mathcal{C}(z)} \int\limits^{z}\omega^{(0)}(x,\cdot )
\, \omega (y,z) = 0,
\]
where $\mathcal C$ is a simple Jordan curve whose interior region contains  $\Con{a}$ for all $a\in \I$.
\end{proof} 
For $k,l\ge 1$ and $a,b\in \I$ we define  1-forms
\begin{align}
L_{a}(k,x):=\frac{\rho_{a} ^{k/2}}{\tpi\sqrt{k}}
\oint\limits_{\mathcal{C}_{a}(z_{a})}z_{a}^{-k}\omega ^{(0)}(x,z_{a})
=\frac{\sqrt{k}\rho_{a}^{k/2}}{(x-w_{a})^{k+1}}dx,  \label{eq:Lkdef}
\end{align}
and  the moment matrix
\begin{align}
A_{ab}(k,l) := &-\frac{\rho_{a}^{k/2}\rho_{b}^{l/2}}{(\tpi)^{2}\sqrt{kl}}
\oint\limits_{\mathcal{C}_{-a}(x)}
\oint\limits_{\mathcal{C}_{b}(y)}x^{-k}y^{-l}\omega
^{(0)}(x,y)  \notag \\
= &
\begin{cases}
\dfrac{(-1)^{k}(k+l-1)!}{\sqrt{kl}(k-1)!(l-1)!}\dfrac{\rho_{a}^{k/2}\rho_{b}^{l/2}}{(w_{-a}-w_{b})^{k+l}},& a\neq -b,
\\
0,& a=-b,
\end{cases}
\label{eq:Adef}
\end{align}%
Let $L(x)=(L_{a}(k,x))$, $R(y):=(L_{-a}(k,y))$ and $A=(A_{ab}(k,l))$ denote the
infinite row vector, column vector and matrix  doubly indexed by  $a,b\in \I$ and $k,l\ge 1$.
Let $I$ denote the infinite doubly indexed identity matrix and $(I-A)^{-1}:=\sum_{n\geq 0}A^{n}$.
We then find
\begin{proposition}
\label{prop:omgsew} 
$\omega (x,y)$ for $x,y\in \D$ is given by 
\begin{align}
\omega (x,y)=\om{0}{}(x,y)-L(x)\left (I-A\right)^{-1}R(y),
\label{eq:omgsew}
\end{align}
where $(I-A)^{-1}$ is convergent for $(\bm{w,\rho})\in \Cg$. 
\end{proposition}
\begin{proof}
We give a proof using arguments  similar to  \cite{Y} and Sections~3.2 and 5.2 of  \cite{MT1}. 
Since  $\omega (x,y)$ is symmetric and applying \eqref{eq:w2rhointeqn} twice we  find
\begin{align*}
\omega (x,y)= &\om{0}{}(x,y)
+\frac{1}{\tpi}\sum_{a\in\I}\,
\oint\limits_{\mathcal{C}_{a}(z_{a})}\int\limits^{z_{a}}\om{0}{}(x,\cdot )
\;\om{0}{}(y,z_{a})\\
&
+\frac{1}{(\tpi)^2}\sum_{a,b\in\I}\,
\oint\limits_{\mathcal{C}_{a}(z_{a})}\oint\limits_{\mathcal{C}_{b}(z_{b})}
\int\limits^{z_{a}}\om{0}{}(x,\cdot )\;
\omega (z_{a},z_{b}) \int\limits^{z_{b}}\om{0}{}(y,\cdot ).
\end{align*}
Applying \eqref{eq:Lkdef}  
and noting that $\mathcal{C}_{a}(z_{a})\sim-\mathcal{C}_{-a}(z_{-a})$  we  find
\begin{align}
\omega (x,y)=\om{0}{}(x,y)- L(x)(I+Y)R(y),
\label{eq:omY}
\end{align}
for moment matrix 
\begin{align}
Y_{ab}(k,l): =-\frac{\rho_{a}^{k/2}\rho_{b}^{l/2}}{(\tpi)^{2}\sqrt{kl}}
\oint\limits_{\mathcal{C}_{-a}(x)}
\oint\limits_{\mathcal{C}_{b}(y)}x^{-k}y^{-l}\omega(x,y),
\label{eq:Ydef}
\end{align}
for $a,b\in \I$ and $k,l \ge 1$. We note that $Y_{ab}(k,l)$ is convergent for $(\bm{w,\rho})\in \Cg$. 
Define the infinite matrix $Y=\left (Y_{ab}(k,l)\right)$.
Taking moments of \eqref{eq:omY} we obtain $Y=A+A(I+Y)A$
which can be recursively solved to find 
\begin{align}
\label{eq:YA}
Y=\sum_{n\ge 1} A^n=(I-A)^{-1}-I.
\end{align} 
Thus \eqref{eq:omY} implies \eqref{eq:omgsew}. Since $Y$ is convergent for $(\bm{w,\rho})\in \Cg$ then  $(I-A)^{-1}$ is also.
\end{proof}

The matrix $A$  and the determinant $\det (I-A)$ defined by 
\begin{align*}
\log \det (I-A):
=\Tr\log(I-A)
=-\sum_{k\ge 1}\frac{1}{k}\Tr A^{k},
\end{align*}
will be of central importance in our later discussions.
\begin{theorem}
\label{theor:detIA} $\det (I-A)$ is non-vanishing and holomorphic on $ \Cg$. 
\end{theorem}
\begin{proof}
	Let  $\Consig{a}$ denote  the circular Jordan curve of radius $\sigma|\rho_{a}|^{\half}$ centred at $w_{a}$ with local coordinate $z_{a}=z-w_{a}$  for some $\sigma>1$ i.e. $\left| z_{-a}z_{a}\right|>|\rho_{a}|$. 
	Consider the sum of integrals
\begin{align*}
S(\bm{w,\rho})=\sum\limits_{a\in \I}\frac{1}{(\tpi)^{2}}
	\oint\limits_{\Consig{-a}}
	\oint\limits_{\Consig{a}}%
	\omega (z_{-a},z_{a})\log \left(1-\frac{\rho_a }{z_{-a}z_{a}}\right).
\end{align*}
We first show that $S(\bm{w,\rho})$ is holomorphic on $\Cg$.
Let $\Delta_{3g}=\{\bm{\mu}:|\mu_{i}|< R_{i}\}\subset\Cg$ be a polydisc for local coordinates $\bm{\mu}=(\mu_{1},\ldots,\mu_{3g})$.   
Write $\omega (z_{-a},z_{a})=f(z_{-a},z_{a},\bm{\mu})dz_{-a}dz_{a}$ for $z_{\pm a}\in \Consig{\pm a}$. 
$f(z_{-a},z_{a},\bm{\mu})$ is holomorphic on $\Cg$ which implies that
\[
f(z_{-a},z_{a},\bm{\mu})=\sum_{\bm{n}}\bm{\mu^{n}}f_{\bm{n}}(z_{-a},z_{a}),
\]
is absolutely convergent on  ${\Delta}_{3g}$ 
where $\bm{\mu^{n}}:=\prod_{i}\mu_{i}^{n_{i}}$ for integers $n_{i}\ge 0$. Furthermore,  $f_{\bm{n}}(z_{-a},z_{a})$ satisfies Cauchy's inequality  e.g.  \cite{Gu}
\begin{align*}
|f_{\bm{n}}(z_{-a},z_{a})|\le \frac{M}{\bm{R^n}},
\end{align*}
for $\bm{R^n}=\prod_{i}R_{i}^{n_{i}}$ and $M=\max_{a}\sup\limits
_{z_{\pm a}\in \Consig{\pm a}}\sup\limits
_{\bm{\mu}\in \Delta_{3g}}|f(z_{a},z_{-a} )|$.
We then find that
\begin{align*}
|S(\bm{w,\rho})|&
\leq \sum\limits_{a\in \I}\sum_{\bm{n}}\frac{\bm{|\mu|^{n}}}{(2\pi )^{2}}
\oint\limits_{\Consig{-a}}
\oint\limits_{\Consig{a}}
\left|f_{n}(z_{-a},z_{a})
\log \left(1-\frac{\rho_a }{z_{-a}z_{a}}\right)dz_{-a}dz_{a}\right| \\
&\leq M\sigma^{2}\left|\log \left(1-{\sigma^{-2}}\right)\right|
\sum\limits_{a\in \I}|\rho_{a}|
\prod_{i}
\left(1-\frac{|\mu_{i}|}{R_{i}}\right)^{-1}.
\end{align*}%
Thus $S$ is absolutely convergent and holomorphic on $\Cg$. 
Since $\left| z_{-a}z_{a}\right|>|\rho_{a}|$ we find
\begin{align*}
S(\bm{w,\rho})= &-\sum\limits_{a\in \I}\sum_{k\geq 1}\frac{\rho_a ^{k}}{k}
\frac{1}{(\tpi)^{2}}
\oint\limits_{\Consig{-a}}
\oint\limits_{\Consig{a}}
\omega (z_{-a},z_{a})z_{-a}^{-k}z_{a}^{-k} \\
= &\Tr Y=\sum_{n\geq 1}\Tr(A^{n}),
\end{align*}%
using \eqref{eq:Ydef} and \eqref{eq:YA}.
Therefore $\sum_{n\geq 1}\Tr(A^{n})$ is  holomorphic  on $\Cg$ and thus it follows that 
$\sum_{n\geq 1}\frac{1}{n}\Tr(A^{n})=-\Tr\log (I-A)$ is also. 
\end{proof}

We identify the  homology cycle $\alpha_{a}$  with $ \Con{-a}$ and   $\beta_{a}$ with a path connecting  $z\in  \Con{a}$ to $z'=\gamma_{a}z\in  \Con{-a}$. 
The $g$ normalized holomorphic one forms $\nu_{b}$,  $b\in \Ip$, can be
expressed in terms of  $L(x),R(x),A$ and moments of $\om{0}{w_{b}-w_{-b}}$ of \eqref{eq:omK0} defined for $a\in \I$  by
\begin{align}
d_{b}^{a}(k)&:=
\begin{cases}
\dfrac{\rho_a ^{k/2}}{\sqrt{k}}
\left(- (w_{-b}-w_{a})^{-k}+(w_{b}-w_{a})^{-k}\right),
& |a|\neq b,
\\  \sgn(a)
 \dfrac{\rho_a ^{k/2}}{\sqrt{k}}(w_{-b\sgn(a)}-w_{a})^{-k},& |a|= b.
\end{cases}
\end{align}%
We let $d_{b}=(d^{a}_{b}(k))$ and $\overline{d}_b=(d_{b}^{-a}(k))$ denote $g$  infinite row and column vectors, respectively, indexed by $a\in \I$ and $k\ge 1$.
We find
\begin{proposition}
\label{prop:nu} 
$\nu_{b}(x)$ for $x\in \D$ and $b\in \Ip$ is given by
\begin{align}\label{eq:nug}
\nu_{b}(x)-\om{0}{w_{b}-w_{-b}}(x)=  -d_{b} \,(I-A)^{-1}\,R(x)
=-L(x)\,(I-A)^{-1}\,\overline{d}_{b}.
\end{align}
\end{proposition}
\begin{proof}
Consider the identity 
\[
\oint\limits_{\mathcal{C}(z)} \omega (x,z) \, \int\limits^{z}\om{0}{w_{b}-w_{-b}}
= 0,
\]
where $\mathcal C$ is a simple Jordan curve whose interior region contains  $\Con{a}$ for all $a\in \I$.
Similarly to Lemma~\ref{lem:inteqn2}, this implies the following generalization of Corollary~5 of  \cite{Y} 
\begin{align*}
\nu_{b}(x) -\om{0}{w_{b}-w_{-b}}(x)= &
\frac{1}{\tpi}%
\sum_{a\in \I}\oint\limits_{\mathcal{C}_{a}(z_a)}\omega (x,z_a)
\left(\int^{z_a}\om{0}{w_{b}-w_{-b}} \right.
 -\sgn(a)\delta_{|a|,b}\log z_a \bigg).
\end{align*}
The result follows by repeating an approach similar to Proposition~\ref{prop:omgsew}.
 \end{proof}

 We may also obtain the genus $g$ period matrix  by generalizing
Lemma~5 of  \cite{Y} to find
\begin{proposition}
\label{prop:period} The genus $g$ period matrix $\Omega_{ab}$ for $a,b\in\Ip$ is given by 
\begin{align}
\tpi \Omega_{ab}= &
\log \left(\frac{\left(w_{a}-w_{b} \right)\left(w_{-a}-w_{-b} \right)}
{\left(w_{-a}-w_{b} \right)\left(w_{a}-w_{-b} \right)}\right)
-d_{a}(I-A)^{-1}\,\overline{d}_{b} , \quad a\neq b, 
\label{eq:Omgab}\\
\tpi \Omega_{aa}= &
\log \left(\frac{-\rho_{a} }{\left(w_{a}-w_{-a} \right)^{2}}\right)
-d_{a}(I-A)^{-1}\,\overline{d}_{a} .
\label{eq:Omgaa}
\end{align}%
\end{proposition}

\section{The Heisenberg Vertex Operator Algebra on $\Sg$}
\subsection{Vertex operator algebras}
Consider a simple Vertex Operator Algebra (VOA) with graded vector space $V=\oplus_{n\ge 0}V_{n}$ and  vertex operators $Y(v,z)=\sum_{n\in\Z} v(n)z^{-n-1}$ for $v\in V$ e.g.  \cite{FHL,FLM,Ka,LL,MT3}. We denote the conformal weight of $v\in V_{n}$ by $\wt(v)=n$. In particular, we highlight the  commutator and associativity identities 
\begin{align}
[u(k),Y(v,z)]=& \sum_{j\ge 0}\binom{k}{j}Y(u(j)v,z)z^{k-j}.
\label{eq:Comm}
\\
(y+z)^{M} Y(u,y+z)Y(v,z) =& (y+z)^{M} Y(Y(u,y)v,z),\quad (M\gg 0),
\label{eq:Assoc}
\end{align}
We assume that $V$ is of CFT type (i.e. $V_{0}=\C\vac$) with a unique symmetric invertible invariant bilinear form $\langle\ ,\ \rangle$ with normalization $\langle\vac,\vac \rangle=1$ where  \cite{FHL,Li}
\begin{align*}
\langle Y(a,z)b,c \rangle =& \left\langle b,Y^{\dagger}(a,z)c \right\rangle,
\end{align*}
for $Y^{\dagger}(a,z)=\sum_{n\in\Z} a^{\dagger}(n)z^{-n-1}:=
Y\left(e^{zL(1)}\left(-z^{-2}\right)^{L(0)}a,z^{-1}\right)$. 
We refer to $\langle\ ,\ \rangle$ as the Li-Zamolodchikov (Li-Z) metric  \cite{MT4}.
For a $V$-basis $\{b\}$, we let $\{\overline{b}\}$ denote the Li-Z dual basis.
If $a\in V_{k}$ is quasi-primary, then
$\langle a(n)b,c \rangle =\langle b,a^{\dagger}(n)c \rangle$  with 
\begin{align*}
a^{\dagger}(n)= (-1)^ka(2k-n-2).
\end{align*}
Thus $a^{\dagger}(n)=-a(-n)$ for $a\in V_{1}$ and $L^\dagger(n)=L(-n)$.

Define the genus zero $n$-point correlation function for $v_{i}$ inserted at $y_{i}$ for $i\in\{1,\ldots,n\}$
\begin{align*}
Z^{(0)}(\bm{v,y}):=Z^{(0)}(v_{1},y_{1};\ldots;v_{n},y_{n})=\langle \vac,\bm{Y(v,y)}\vac\rangle,
\end{align*}
where  $\langle\cdot,\cdot\rangle$ is the Li-Z metric and $\bm{Y(v,y)}:= Y(v_{1},y_{1}) \ldots Y (v_{n},y_{n})$. 
$Z^{(0)}(\bm{v,y})$  is a rational function of $y_{i}$ e.g.   \cite{FHL,Z, TW}. 

\subsection{Heisenberg genus zero correlation functions}
In order to later apply the MacMahon Master Theorems~\ref{theor:MMT}--\ref{theor:MMTsubpperm}, 
we consider the rank two Heisenberg VOA $M^{2}$ generated by two weight 1 commuting Heisenberg vectors $h_{1}, h_{2}$. 
Define $h_{\pm}:=\frac{1}{\sqrt{2}}(h_{1}\pm \im h_{2})$ with non-trivial commutator relation
\begin{align}
[h_{+}(k),h_{-}(l)]=k\delta_{k,-l},\quad k,l\in\Z.
\label{eq:Heiscom}
\end{align}
$M^{2}$ has a Fock basis with elements 
\begin{align}
h(\bi_{+},\bi_{-}):=h_{+}{(-1)}^{r_{+}(1)} h_{+}{(-2)}^{r_{+}(2)} \ldots h_{-}{(-1)}^{r_{-}(1)} h_{-}{(-2)}^{r_{-}(2)}\ldots  \vac,
\label{eq:Fock}
\end{align}
labelled by a pair of  multisets 
\begin{align*}
\bi_{+}&=\{1^{r_{+}(1)}2^{r_{+}(2)}\ldots i_{+}^{r_{+}(i_{+})}\ldots \},
\quad
\bi_{-}=\{1^{r_{-}(1)}2^{r_{-}(2)}\ldots i_{-}^{r_{-}(i_{-})}\ldots \},
\end{align*}
where $i_{+}$ occurs $r_{+}(i_{+})$ times in $\bi_{+}$ and $i_{-}$ occurs $r_{-}(i_{-})$ times in $\bi_{-}$. Where no ambiguity arises, we will omit the $\pm$ subscript in $r_{\pm}$. 
The Fock vector has conformal weight
\begin{align}
\label{eq:wthkl}
\wt(h(\bi_{+},\bi_{-}))=\sum_{i_{+}\in\bi_{+}}i_{+}r\left(i_{+}\right)+\sum_{i_{-}\in\bi_{-}}i_{-}r\left(i_{-}\right).
\end{align}
The Fock dual basis with respect to the Li-Z metric  has elements
\begin{align}
\overline{h}{(\bi_{+},\bi_{-})} = \left(
\prod_{i_{+}\in\bi_{+}} \prod_{i_{-}\in\bi_{-}}
\left(\frac{-1}{i_{+}}\right)^{r\left(i_{+}\right)}
\left(\frac{-1}{i_{-}}\right)^{r\left(i_{-}\right)}
\frac{1}{r\left(i_{+}\right)! r\left(i_{-}\right)!} \right)
 h{(\bi_{-},\bi_{+})}.
\label{eq:Fockdual}
\end{align}
The basic genus zero Heisenberg $2$-point function found by applying \eqref{eq:Comm} and \eqref{eq:Heiscom} is\footnote{\label{foot:conv} Here, and below, we adopt the standard convention that 
	$
	(x+y)^{\kappa}:=\sum_{m\ge 0}\binom{\kappa}{m}x^{\kappa-m}y^{m},
	$
	for any $\kappa$ i.e. we formally expand in the second parameter $y$.}
\begin{align}
Z^{(0)}(h_{+},x^{+};h_{-},x^{-})=\frac{1}{(x^{+}-x^{-})^2}.
\label{eq:2pt}
\end{align}
\eqref{eq:2pt} is fundamental to finding the Heisenberg partition and correlation functions on $\Sg$. The genus zero $2n$-point function for $h_{\pm}$ inserted at $x^{\pm}_{i}$ for 
$i\in \{ 1,\ldots,n\}$  is found by using \eqref{eq:Comm} for $h_{+}(k)$ with $k\ge 0$ and \eqref{eq:Heiscom}  via the recursion formula e.g.  \cite{Z,TW} 
\begin{align*}
Z^{(0)}(\bm{h_{\pm},x^{\pm}}) =&
\sum_{j=1}^{n}\frac{1}{(x^{+}_{1}-x^{-}_{j})^2}
Z^{(0)}(h_{+},x^{+}_{2};\ldots;\widehat{h_{-},x^{-}_{j}};\ldots),
\end{align*}
where  the $h_{-}$ insertion at $x^{-}_{j}$ is deleted. Repeating we find 
\begin{align}
Z^{(0)}(\bm{h_{\pm},x^{\pm}}) = 
\perm\frac{1}{(x^{+}_{i}-x^{-}_{j})^2}.\quad i,j\in\{ 1,\ldots,n\}.
\label{eq:2npt0}
\end{align}
Defining $\F^{(0)}(\bm{h_{\pm},x^{\pm}}) :=
Z^{(0)}(\bm{h_{\pm},x^{\pm}}) \bm{dx^{+}dx^{-}}$ with
$\bm{dx^{+}dx^{-}}:=\prod_{i=1}^{n}dx^{+}_{i }dx^{-}_{i} $, 
we may re-express \eqref{eq:2npt0} in terms of differential forms with 
\begin{align}
\F^{(0)}(\bm{h_{\pm},x^{\pm}}) =
\perm\omega^{(0)}(x^{+}_{i},x^{-}_{j}).
\label{eq:form2npt0}
\end{align}

$Z^{(0)}(\bm{h_{\pm},x^{\pm}})$ is a generating function for all genus zero correlation functions for $M^{2}$  in the sense that we may extract any genus zero correlation function as the coefficient of an appropriate expansion of  $Z^{(0)}(\bm{h_{\pm},x^{\pm}})$  e.g.  \cite{MT3,HT}. This follows from the general observation that for $a,b\in V$ we have by   \eqref{eq:Assoc} that
\begin{align}\notag
Z^{(0)}( \ldots;a(-k)b,w;\ldots)=&\coeff_{x^{k-1}}Z^{(0)}( \ldots;Y(a,x)b,w;\ldots)
\\
\label{eq:genrel}
=&\coeff_{x^{k-1}} Z^{(0)}( \ldots; a,x+w; b,w;\ldots),
\end{align}
where we may omit any $(x+w)^{M}$ factors when comparing rational functions.
Thus 
$Z^{(0)}(\ldots;{h}{(\bi_{+},\bi_{-})},w;\ldots )$ is the coefficient of 
$\prod_{i_{\pm}}\prod_{p=1}^{r\left(i_{+}\right)}\prod_{q=1}^{r(i_{-})}
(x^{+}_{i_{+},p})^{i_{+}-1}(x^{-}_{i_{-},q})^{i_{-}-1}
$
in  $Z^{(0)}\left( \ldots ;h_{+},x^{+}_{i_{+},p}+w; \ldots ;h_{-},x^{-}_{i_{-},q}+w;\ldots \right)$ 
 using $Y(\vac,w)=1$. 
We consider below genus zero correlation functions for vectors inserted at $w_{a}$ of \eqref{eq:sew} for $a\in\I$.
Let $y=x^{+}-w_{-a}$ and $z=x^{-}-w_{b}$ be local coordinates in the neighborhood of $w_{-a},w_{b}$. 
For the  two point function appearing in \eqref{eq:2pt} and \eqref{eq:2npt0} we define
\begin{align}
\Ah_{ab}(k,l)&=\coeff_{ y^{k-1} z^{l-1}} 
\frac{1}{(x^{+}-x^{-})^{2}}
=-\rho_{a}^{-k/2}\rho_{b}^{-l/2}\sqrt{kl}A_{ab}(k,l),
\label{eq:Ahat}
\end{align}
for the moment matrix  $A$ of  \eqref{eq:Adef}. 

\subsection{The genus $g$ partition function on $M^{2}$}
Let $\bm{b}=(b_{1},\ldots,b_{g})$ denote an element of a $V^{\otimes g}$-basis with Li-Z dual  $\bm{\overline{b}}=(\bbar_{1},\ldots,\bbar_{g})$ and consider the rational genus zero $2g$-point correlation function  for these vectors inserted at $w_{a}$ of \eqref{eq:sew}
\begin{align*}
Z^{(0)}(\bm{b,w}):=Z^{(0)}(\bbar_{1},w_{-1};b_{1},w_{1};\ldots;\bbar_{g},w_{-g};b_{g},w_{g}).
\end{align*}
We  define the genus $g$ partition function for a VOA $V$ by  \cite{TW} 
\begin{align}
\label{eq:ZVg}
Z_{V}^{(g)}(\bm{w,\rho})
:=\sum_{\bm{b}\in V^{\otimes g}}\bm{\rho^{\wt(b)}}Z^{(0)}(\bm{b,w}),
\end{align}
where $\bm{w,\rho}:=w_{\pm 1},\rho_{1},\ldots ,w_{\pm g},\rho_{g}$ and
$\bm{\rho^{\wt(b)}}:=\prod_{1\le a\le g}\rho_{a}^{\wt(b_{a})}$ for Schottky  parameters  $\rho_{a}$ of \eqref{eq:sew}. In general $Z_{V}^{(g)}(\bm{w,\rho})$ is a formal series in $\bm{\rho}$ but with convergent coefficients.
\begin{proposition}
	\label{prop:ZM}
	The  genus $g$ partition function for the Heisenberg VOA $M^{2}$  is
	\[
	Z_{M^{2}}^{(g)}(\bm{w,\rho})=\det(I-A)^{-1}, 
	\]
	for the moment matrix $A$ of \eqref{eq:Adef}.
\end{proposition}
\begin{proof}
Here we sum over the Fock basis vectors $b_{a}=h(\bi_{a},\bi_{-a})$  labelled by $2g$ multisets
\begin{align*}
\bi_{a}:=\{\ldots i_{a}^{r(i_{a})}\ldots \}, \quad a\in\I,
\end{align*} 
i.e. $i_{a}$ occurs $r(i_{a})$ times in $\bi_{a}$.
It follows from \eqref{eq:wthkl} and   \eqref{eq:Fockdual} that
\begin{align*}
\bm{\rho^{\wt(b)}}Z^{(0)}(\bm{b,w})=&\frac{M(\bm{\rho},\bi_{a})}{\brfact}
 Z^{(0)}\left(\ldots ;h(\bi_{a},\bi_{-a}),w_{a};\ldots \right),
\end{align*}
where 
\begin{align}
\label{eq:M}
\brfact:=\prod_{a\in\I}\prod_{i_{a}}r(i_{a})!,
\quad
M(\bm{\rho},\bi_{a}):=
\prod_{a\in\I}\prod_{i_{a}}
\left(\frac{-\rho_{a}^{i_{a}}}{i_{a}}\right)^{r(i_{a})}. 
\end{align}
$Z^{(0)}\left(\ldots ;h(\bi_{a},\bi_{-a}),w_{a};\ldots \right)$ is labelled by 
$\bi_{a}$ for $ a\in\I$ and  is determined by the coefficient of an appropriate expansion of the generating function \eqref{eq:2npt0} as described above. Hence
\begin{align*}
 Z^{(0)}\left(\ldots ;h(\bi_{a},\bi_{-a}),w_{a};\ldots \right)
 =\perm \Ah\left(\bi_{a},\bi_{a}\right),
\end{align*}
for $\Ah$ of \eqref{eq:Ahat} where  $\Ah\left(\bi_{a},\bi_{a}\right)$ is an $N\times N$  matrix  for $N=\sum_{a\in\I}\sum_{i_{a}}r(i_{a})$.
The   $-\rho_{a}^{i_{a}}/{i_{a}}$ multiplicative factors in $M(\bm{\rho},\bi_{a})$ can be absorbed into the permanent to obtain 
\begin{align*}
Z_{M^{2}}^{(g)}(\bm{w,\rho})=\sum_{\bi_{a}}\frac{\perm A\left(\bi_{a},\bi_{a}\right)}{\brfact},
\end{align*}
for moment matrix $A$ and where the sum is taken over all multisets $\bi_{a}$.
We may truncate  the formal series $Z_{M^{2}}^{(g)}(\bm{w,\rho})$ to any finite order in $\bm{\rho}$ and apply the MacMahon Master Theorem ~\ref{theor:MMT} to find that  $Z_{M^{2}}^{(g)}(\bm{w,\rho})=\det(I-A)^{-1}$ to any finite order in $\bm{\rho}$ and therefore to all orders since $\det(I-A)^{-1}$ is convergent by Theorem~\ref{theor:detIA}. Hence the result follows. 
\end{proof}
From the definition \eqref{eq:ZVg} we have  $Z_{U\otimes V}^{(g)}=Z_{U}^{(g)}Z_{V}^{(g)}$ for any two VOAs $U$ and $V$  \cite{TW}. Thus we have  $Z_{M^{2}}^{(g)}=(Z_{M}^{(g)})^{2}$ so that  
\begin{align}
\label{eq:ZM1g}
Z_{M}^{(g)}=\det(I-A)^{-\half}.
\end{align}
\eqref{eq:ZM1g} generalizes results of  \cite{MT4, MT5} for genus 2. 
 In general,  $Z_{V}^{(g)}$  is M\"obius invariant by Proposition~4.3 of   \cite{TW}.  Thus together with  Theorem~\ref{theor:detIA} it follows that 
\begin{corollary}
$Z_{M^{2}}^{(g)}(\bm{w,\rho})=\det(I-A)^{-1}$ is holomorphic on  Schottky space $\Schg$.
\end{corollary}
\subsection{Montonen-Zograf product formula}
$\det(I-A)$ can be expressed in terms of an infinite product formula originally discovered by Montonen in 1974  \cite{Mo} (see  \cite{DV} also). This formula was subsequently found in the holomorphic part of a Laplacian determinant formula on $\Sg$ by Zograf  \cite{Z} related to earlier work of D'Hoker and Phong  \cite{DP}.

$\gamma_{p}\in\Gamma $ (the Schottky group)  is called primitive if $\gamma_{p}\neq \gamma^{m}$ for any $m>1$ and $\gamma\in \Gamma$. 
Thus  $\gamma=\gamma_{p}^{m}$ for some primitive $\gamma_{p}$ and  some $m\ge 1$ for every $\gamma\in\Gamma$. We then have
\begin{theorem}
\[	
\det(I-A)=\prod_{k\ge 1}\prod_{\gamma_{p}}(1-q_{\gamma_{p}}^{k})
,\]
where $\gamma_{p}$ is summed over representatives of the primitive conjugacy classes of $\Gamma$ excluding the identity and  $q_{\gamma_{p}}$ is the multiplier of $\gamma_{p}$.
\end{theorem}
\begin{proof}
The proof is a variation of arguments in   \cite{Mo} and   \cite{DV}.	
Recall that $\log\det(I-A)=-\sum_{n\ge 1}\frac{1}{n}\Tr A^{n}$. Using \eqref{eq:aA}, of the Appendix, we have
\begin{align*}
\Tr A^{n}&=\sum'_{a_{1},\ldots,a_{n}\in\I}\Tr \left( D(\mu_{a_{1}}\lambda_{a_{2}}^{-1})
D(\mu_{a_{2}}\lambda_{a_{3}}^{-1})\ldots D(\mu_{a_{n}}\lambda_{a_{1}}^{-1})\right),
\end{align*}
for $\lambda_{a},\mu_{a}$ of \eqref{eq:lammu} where the prime indicates that $-a_{i}\neq a_{i+1}$ for $i=1,\ldots,n-1$ and $-a_{n}\neq a_{1}$. Note that the summand trace is over the integer labels of the  $D(\gamma)$ matrices only. From Lemma~\ref{lem:Drep}~(iii) and \eqref{eq:gammaSchot} it follows that\footnote{One has to check that the conditions of Lemma~\ref{lem:Drep} are satisfied.}
\begin{align*}
\Tr A^{n}&=\sum'_{a_{1},\ldots,a_{n}\in\I}\Tr D\left(\lambda_{a_{1}}^{-1}\mu_{a_{1}}\lambda_{a_{2}}^{-1}
\mu_{a_{2}}\ldots \lambda_{a_{n}}^{-1}\mu_{a_{n}}\right)
\\
&=\sum'_{a_{1},\ldots,a_{n}\in\I}
\Tr D
\left(\gamma_{a_{1}}\gamma_{a_{2}}\ldots \gamma_{a_{n}}
\right)
= \sum_{\gamma\in \Gamma_{n}^{\text{CR}}}\Tr D(\gamma),
\end{align*}
 where $\Gamma_{n}^{\text{CR}}$ is the set of Cyclically Reduced words of length $n$ in $\Gamma$ i.e.  reduced words $\gamma_{a_{1}}\ldots \gamma_{a_{n}}$ for which $\gamma_{a_{1}}^{-1}\equiv\gamma_{-a_{1}}\neq \gamma_{a_{n}}$.  
For $\gamma\in \Gamma_{n}^{\text{CR}}$ we have $\gamma=\gamma_{p}^{m}$ for some primitive cyclically reduced word $\gamma_{p}$ and some $m\ge 1$ where $m|n$. 
Every element of  $\Gamma$ is conjugate to a cyclically reduced word and any two cyclically reduced words are conjugate if and only if they are cyclic permutations of each other e.g. Prop.~9 of  \cite{C}. 
Then it follows that there are $n/m$ cyclically reduced words conjugate to $\gamma_{p}^{m}$. Therefore we find
\begin{align*}
\sum_{n\ge 1}\frac{1}{n}\Tr A^{n}=\sum_{\gamma_{p}}\sum_{m\ge 1}\frac{1}{m}\Tr D(\gamma_{p}^{m}),
\end{align*}
where $\gamma_{p}$ ranges over representatives of the primitive conjugacy classes  of $\Gamma$ excluding the identity. 
$\gamma_{p} $ is conjugate in $\SL(2,\C)$ to $\diag(q_{\gamma_p}^{1/2},q_{\gamma_{p}}^{-1/2})$ for multiplier $q_{\gamma_{p}}$.
From \eqref{eq:Dkl} we thus find  $\Tr \left(D_{kl}(\gamma_{p}^{m})\right)=\Tr \left( \delta_{kl}q_{\gamma_{p}}^{mk}  \right)=\sum_{k\ge 1}q_{\gamma_{p}}^{mk}$  so that
\begin{align*}
\log\det(I-A)&=-\sum_{k\ge 1}\sum_{\gamma_{p}}\sum_{m\ge 1}\frac{1}{m}q_{\gamma_{p}}^{mk}=\sum_{k\ge 1}\sum_{\gamma_{p}}\log(1-q_{\gamma_{p}}^{k}).
\end{align*}
\end{proof}
\subsection{Genus $g$ correlation functions on $M^{2}$}
We define genus $g$  formal $n$-point correlation differential forms for $n$ vectors $v_{1},\ldots,v_{n}\in V$ inserted at $y_{1},\ldots,y_{n}$ by  \cite{TW}
\begin{align}\label{eq:GenusgnPoint}
\F_{V}^{(g)}(\bm{v,y}):=
\sum_{\bm{b}\in V^{\otimes g}}\bm{\rho^{\wt(b)}}Z^{(0)}(\bm{v,y};\bm{b,w})\bm{dy^{\wt(v)}},
\end{align}
where $
Z^{(0)}(\bm{v,y};\bm{b,w})=Z^{(0)}(v_{1},y_{1};\ldots;v_{n},y_{n};\bbar_{1},w_{-1};b_{1},w_{1};\ldots;\bbar_{g},w_{-g};b_{g},w_{g})$.  Consider  $\F_{M^{2}}^{(g)}(\bm{\bm{h_{\pm},y^{\pm}}})$ for Heisenberg generators $h_{\pm}$ inserted at $y_{r}^{\pm}$ for $r\in \{1,\ldots,m\}$ which is the generating function for all correlation functions just as \eqref{eq:form2npt0} is at genus zero. 
\begin{proposition}
	\label{prop:ZMnptg}
The genus $g$ generating function for  $M^{2}$ is given by 
\begin{align*}
\F_{M^{2}}^{(g)}(\bm{\bm{h_{\pm},y^{\pm}}})=
\frac{\perm \omega\left(y^{+}_{r},y^{-}_{s}\right) } {\det(I-A)},\quad r,s\in\{ 1,\ldots ,m\},
\end{align*} 
for bidifferential form $\omega(x,y)$ of \eqref{eq:omegag}.
\end{proposition}
\begin{proof}
The $\bm{b}$ sum of \eqref{eq:GenusgnPoint} is taken over $h(\bi_{a},\bi_{-a})$ with summand
\begin{align*}
&\bm{\rho^{\wt(b)}}Z^{(0)}(\bm{\bm{h_{\pm},y^{\pm}};b,w})=
\frac{M(\bm{\rho},\bi_{a})}{\brfact}X(\bi_{a}),
\end{align*}
for $X(\bi_{a})=Z^{(0)}\left(\ldots;h_{+},y^{+}_{r};\ldots;h_{-},y^{-}_{s} \ldots ;h(\bi_{a},\bi_{-a}),w_{a};\ldots \right)$ determined by an expansion of
the generating function \eqref{eq:2npt0} given by
\begin{align*}
X(\bi_{a})=\perm 
\begin{bmatrix}
\Bh&  \Uh(\bi_{a})\\ 
\Vh(\bi_{a}) & \Ah\left(\bi_{a},\bi_{a}\right)
\end{bmatrix},
\end{align*}
	for $\Ah$ of \eqref{eq:Ahat} and where for $r,s\in \{1,\ldots,m\}$, $ i_{a}\in \bi_{a}$, $ j_{b}\in \bi_{b}$ we define
\begin{align}
\label{eq:BUVhat}
\Bh(r,s):=\frac{1}{(y^{+}_{r}-y^{-}_{s})^{2}},\quad \Uh(r,j_{b}):=\frac{j_{b}}{(y^{+}_{r}-w_{b})^{j_{b}+1}},
\quad \Vh(i_{a},s):=\frac{i_{a}}{(y^{-}_{s}-w_{-a})^{i_{a}+1}}.
\end{align}
The additional  $-\rho_{a}^{i_{a}}/{i_{a}}$ and $\bm{dy^{+}_{r}dy^{-}_{s}}$ factors can be absorbed into the permanent to obtain 
\begin{align*}
\F_{M^{2}}^{(g)}(\bm{\bm{h_{\pm},y^{\pm}}})=\sum_{\bi_{a}}\frac{1}{\brfact}
\perm
\begin{bmatrix}
B&  U(\bi_{a})\\ 
V(\bi_{a}) & A\left(\bi_{a},\bi_{a}\right)
\end{bmatrix},
\end{align*}
where  with  $L,R$ of \eqref{eq:Lkdef} we define
\begin{align}
\label{eq:BUV}
B(r,s):=\omega^{(0)}(y^{+}_{r},y^{-}_{s}), \quad
U(r,j_{b}):=L_{a}(j_{b},y_{r}^{+}), \quad
V(i_{a},s):=-R_{a}(i_{a},y_{s}^{-}), 
\end{align}
Applying Theorem~\ref{theor:MMTsub} we find $
\widetilde{B}=B+U(I-A)^{-1}V$ is given by 
\begin{align}
\label{eq:Btilde}
\widetilde{B}(r,s)=\omega^{(0)}(y^{+}_{r},y^{-}_{s})-L(y^{+}_{r}) (I-A)^{-1}R(y^{-}_{r})
=\omega(y^{+}_{r},y^{-}_{s}),
\end{align}
by Proposition~\ref{prop:omgsew}. Thus the result holds.
\end{proof}

\section{The Heisenberg Generalized VOA on $\Sg$}
\subsection{The Heisenberg generalized VOA $\M^{2}$}
The Heisenberg VOA $M^{2}$, generated by $h_{\pm}$, has irreducible modules $M^{2}_{\alpha}=M^{2}\otimes e^{\alpha}$ for $\alpha\in \C^{2}$
where for $u\otimes e^{\alpha}\in M^{2}_{\alpha}$ 
\begin{align}
h_{\pm}(0)(u\otimes e^{\alpha}) =& \alpha_{\pm} (u \otimes e^{\alpha}),
\qquad 
h_{\pm}(n)(u\otimes e^{\alpha}) =  (h_{\pm}(n)u) \otimes e^{\alpha}, \ n\neq 0, 
\label{eq:alphan}
\end{align}
where 
$\alpha_{\pm}:=\frac{1}{\sqrt{2}}(\alpha_{1}\pm \im \alpha_{2})$.
In  \cite{TZ} an intertwiner vertex operator $\calY(u\otimes e^{\alpha},z)$ which creates  $u\otimes e^{\alpha}$ from $\vac$ is defined, similarly for  lattice vertex operators  \cite{FLM,Ka},  by
\begin{align*}
\calY(u\otimes e^{\alpha},z):=& e^{\alpha} Y_{-}(\alpha, z)  
Y(u ,z)Y_{+}(\alpha,z)z^{\alpha(0)},
\\
Y_{\pm}(\alpha, z) :=& \exp \left(\mp \;  \sum_{n>0}\frac{\alpha(\pm \; n)}{n} z^{\mp n}\right), 
\end{align*} 
where $\alpha(n):=\sum_{i=1}^{2}\alpha_{i} h_{i}(n)$. $e^{\alpha}$ is a twisted group algebra element for the abelian group $\C^{2}$. The twisted group algebra  is  associative with 2-cocycle $\varepsilon(\alpha,\beta)\in \C^{\times}$ where
\begin{align*}
e^{\alpha}e^{\beta}=\varepsilon(\alpha,\beta)e^{\alpha+\beta},\quad e^{0}=1,
\end{align*}  
so that $\varepsilon(\alpha,0)=\varepsilon(0,\alpha)=1$. 
Associativity implies that the commutator 
\begin{align*}
C(\alpha,\beta):=\varepsilon(\alpha,\beta)\varepsilon(\beta,\alpha)^{-1},
\end{align*}
 is skew-symmetric and multiplicatively bilinear. For a given commutator, we may choose cocycles  such that  $\varepsilon(\alpha,-\alpha)=1$  for all $\alpha\in\C^{2}$  \cite{TZ}.

 Define the vector space $\M^{2}:=\oplus_{\alpha\in \C^{2}}M^{2}_{\alpha}$. 
 For a given commutator $C(\alpha,\beta)$, the vertex operators on $\M^{2}$ form a generalized VOA\footnote{The generalized VOA is of a more general type than those described in  \cite{DL}.} 
with 
 $\wt\left( u\otimes e^{\alpha}\right)=\wt(u)+\alpha_{+}\alpha_{-}\in\C$ for the Heisenberg Virasoro vector  $\omega=h_{+}(-1)h_{-}$ \cite{BK,TZ}.   
  In particular, for all $u\in M^{2}$ and $v\otimes e^{\alpha}\in M^{2}_{\alpha}$ we obtain the following  natural commutator and associativity identities
\begin{align}
[u(k),\calY(v\otimes e^{\alpha},z)]=& \sum_{j\ge 0}\binom{k}{j}\calY(u(j)v\otimes e^{\alpha},z)z^{k-j},
\label{eq:Comm_alpha}
\\
(y+z)^{N} Y(u,y+z)\calY(v\otimes e^{\alpha},z) =& (y+z)^{N} \calY(Y(u,y)v\otimes e^{\alpha},z),\quad (N\gg 0),
\label{eq:Assoc_alpha}
\end{align}
where we identify $\calY(u\otimes e^{0},z)$ with $Y(u,z)$.

Let $\alpha^{1},\ldots,\alpha^{k}\in\C^{2}$ and consider the genus zero correlation function
\begin{align*}
Z^{(0)}(\bm{e^{\alpha},z}) := \langle \vac, \calY(e^{\alpha^{1}},z_{1})\ldots \calY(e^{\alpha^{k}},z_{k})\vac \rangle,
\end{align*}
abbreviating $\vac\otimes e^{\alpha}$ by  $e^{\alpha}$.
Using the standard lattice VOA identity\footnote{See footnote~\ref{foot:conv}}   \cite{FLM}
\begin{align*}
Y_{+}(\alpha,x)Y_{-}(\beta,y) =& \left(1-\frac{y}{x}  \right)^{\alpha\cdot\beta}Y_{-}(\beta,y) Y_{+}(\alpha,x),
\end{align*}
where we define 
$\alpha\cdot \beta:=\alpha_{+}\beta_{-}+\alpha_{-}\beta_{+}$ for all $\alpha,\beta \in \C^{2}$.
We find that 
$Z^{(0)}(\bm{e^{\alpha},z}) = 0$ for $\sum_{t=1}^{k}\alpha^{t}\neq 0$ whereas
for $\sum_{t=1}^{k}\alpha^{t}= 0$ we have 
\begin{align}
\label{eq:Z0alpha}
Z^{(0)}(\bm{e^{\alpha},z}) = \varepsilon_{\bm{\alpha}}\prod_{1\le t<u\le k}
(z_{t}-z_{u})^{\alpha^{t}\cdot \alpha^{u}},
\end{align}
where 
\begin{align}
\varepsilon_{\bm{\alpha}}:=\prod_{t=1}^{k-1}\varepsilon\left(\alpha^{t},\sum_{u=t+1}^{k}\alpha^{u}\right).
\label{eq:epsalpha}
\end{align}
For Heisenberg generators $h_{\pm}$ inserted at $x_{i}^{\pm}$ for $i\in \{1,\ldots,n\}$ and  $e^{\alpha^{t}}$ inserted at $z_{t}$ for $t\in \{1,\ldots,k\}$ where $\sum_{t=1}^{k}\alpha^{t}= 0$, 
consider the  genus zero correlation function
\begin{align}
\label{eq:Z0halpha}
Z^{(0)}(\bm{h_{\pm},x^{\pm}};\bm{e^{\alpha},z}) = 
\langle \vac, Y(h_{+},x_{1}^{+})\ldots Y(h_{-},x_{n}^{-})
\calY(e^{\alpha^{1}},z_{1})\ldots \calY(e^{\alpha^{k}},z_{k})\vac \rangle.
\end{align}
\eqref{eq:Z0halpha} is a generating function for all  $\M^{2}$ genus zero correlation functions much as $Z^{(0)}(\bm{h_{\pm},x^{\pm}})$ of \eqref{eq:2npt0} is for $M^{2}$. 
Using \eqref{eq:Comm}, \eqref{eq:Heiscom}, \eqref{eq:alphan} and \eqref{eq:Comm_alpha} we find
that 
\begin{align*}
Z^{(0)}(\bm{h_{\pm},x^{\pm}};\bm{e^{\alpha},z}) =
&
\sum_{t=1}^{k}\frac{\alpha_{+}^{t}}{x^{+}_{1}-z_{t}}
Z^{(0)}(h_{+},x^{+}_{2};\ldots;\bm{e^{\alpha},z}),
\\
&+
\sum_{j=1}^{n}\frac{1}{(x^{+}_{1}-x^{-}_{j})^2}
Z^{(0)}(h_{+},x^{+}_{2};\ldots;\widehat{h_{-},x^{-}_{j}};\ldots;\bm{e^{\alpha},z}).
\end{align*}
Repeating for each $h_{\pm}$  and using \eqref{eq:Z0alpha} we obtain a partial permanent   (cf.  \eqref{eq:pperm}) formula:
\begin{lemma}
	For $\sum_{t=1}^{k}\alpha^{t}= 0$ we have
	\begin{align} 
	Z^{(0)}(\bm{h_{\pm},x^{\pm}};\bm{e^{\alpha},z})  = 
	\varepsilon_{\bm{\alpha}}\prod_{1\le t<u\le k}
	(z_{t}-z_{u})^{\alpha^{t}\cdot \alpha^{u}}
	\;\pperm_{\,\theta^{-},\theta^{+}}\frac{1}{(x^{+}_{i}-x^{-}_{j})^2},
	\label{eq:2nptalpha}
	\end{align}
	where  
	\begin{align}
	\label{eq:thetapmi}
\theta^{\pm}_{i}&=\sum_{t=1}^{k}\frac{\alpha_{\pm}^{t}}{x^{\pm}_{i}-z_{t}},\quad 
i\in \{1,\ldots,n\}.
\end{align} 
\end{lemma}
Let 
$\F^{(0)}(\bm{h_{\pm},x^{\pm}};\bm{e^{\alpha},z}):=
Z^{(0)}(\bm{h_{\pm},x^{\pm}};\bm{e^{\alpha},z})
\bm{dx^{+}dx^{-}dz^{\half \alpha^2}}$ for 
$\bm{dz^{\half \alpha^2}}=\prod_{t=1}^{k}dz_{t}^{\alpha^{t}_{+}\alpha^{t}_{-}} $, much as in \eqref{eq:form2npt0}. For genus zero differential of the third kind $\omega^{(0)}_{p-q}$ of \eqref{eq:omK0}
 we find
\begin{align*}	
\chi^{\pm}_{i}:=\theta^{\pm}_{i}dx_{i}^{\pm}=\sum_{t=1}^{k}\alpha_{\pm}^{t}
\omega^{(0)}_{z_{t}-z_{0}}(x^{\pm}_{i}),
\end{align*} 
where $\chi^{\pm}_{i}$ is independent of $z_{0}$ since $\sum_{t=1}^{k}\alpha^{t}_{\pm}= 0$.
Then \eqref{eq:2nptalpha} may be rewritten in terms of the genus zero differentials 
$\omega^{(0)}(x,y)$, $E^{(0)}(x,y)$ and $\omega^{(0)}_{p-q}(x)$ as follows:
\begin{corollary}
	\label{cor:Zohalpha}
	\begin{align} 
	\F^{(0)}(\bm{h_{\pm},x^{\pm}};\bm{e^{\alpha},z})  = 
	\varepsilon_{\bm{\alpha}}\prod_{1\le t<u\le k}
	E^{(0)}(z_{t},z_{u})^{\alpha^{t}\cdot \alpha^{u}}
	\;\pperm_{\,\chi^{-},\chi^{+}}\omega^{(0)}(x^{-}_{i},x^{+}_{j}).
	\label{eq:form2nptalpha0}
	\end{align}
\end{corollary}
$M^{2}_{\alpha}$ has a Fock basis with elements 
\begin{align}
\label{eq:Fockalpha}
h^{\alpha}(\bi_{+},\bi_{-}):=\ldots h_{+}{(-i_{+})}^{r(i_{+})}\ldots h_{-}{(-i_{-})}^{r(i_{-})} \ldots  \vac\otimes e^{\alpha},
\end{align}
labelled by a pair of  multisets $\bi_{+},\bi_{-}$ and with conformal weight
\begin{align*}
\wt(h^\alpha(\bi_{+},\bi_{-}))=\sum_{i_{+}}i_{+}r\left(i_{+}\right)+\sum_{i_{-}}i_{-}r\left(i_{-}\right)+\half\alpha\cdot\alpha.
\end{align*}
A Li-Z metric exists on $\M^{2}$ leading to a Fock dual basis with elements  \cite{TZ}  (with  cocycle choice $\varepsilon(\alpha,-\alpha)=1$) given by 
\begin{align}
\overline{h^\alpha}{(\bi_{+},\bi_{-})} =
(-1)^{\half \alpha\cdot\alpha } \left(\prod_{i_{\pm}}\left(\frac{-1}{i_{+}}\right)^{r\left(i_{+}\right)}
\left(\frac{-1}{i_{-}}\right)^{r\left(i_{-}\right)}
\frac{1}{r\left(i_{+}\right)! r\left(i_{-}\right)!} \right)
h^{-\alpha}{(\bi_{-},\bi_{+})}.
\label{eq:Fockdualalpha}
\end{align}
\subsection{Genus $g$ partition functions on $\M^{2}$}
Let $\alpha^{a} \in\C^{2}$ for $a\in\Ip$ and consider  $M^{2}_{\bm{\alpha}}:=\otimes_{a\in\Ip}M^{2}_{\alpha^{a}}$.
Similarly to \eqref{eq:ZVg}, we define the genus $g$ partition function for $M^{2}_{\bm{\alpha}}$  by
\begin{align}
\label{eq:ZVgalpha}
Z_{M^{2}_{\bm{\alpha}}}^{(g)}(\bm{w,\rho})
:=\sum_{\bm{b^{\alpha}}\in M^{2}_{\bm{\alpha}}}\bm{\rho^{\wt(b^{\alpha})}}Z^{(0)}(\bm{b^{\alpha},w}),
\end{align}
for 
$Z^{(0)}(\bm{b^{\alpha},w}):=Z^{(0)}(\ldots;\overline{b_{a}^{\alpha^{a}}},w_{-a};b^{\alpha^{a}}_{a},w_{a};\ldots)$.  
Define $\alpha^{-a}:=-\alpha^{a}$ for $a\in\Ip$ so that  
$\overline{\vac\otimes e^{\alpha^{a}}}=(-1)^{\half \alpha^{a}\cdot\alpha^{a} } \vac\otimes e^{\alpha^{-a}}$. From \eqref{eq:Z0alpha} and \eqref{eq:epsalpha} we find  that  $\varepsilon_{\bm{\alpha}}=1$  and
\begin{align}
\label{eq:Z0lwalpha}
\langle \vac, \ldots\calY(e^{\alpha^{-a}},w_{-a})\calY(e^{\alpha^{a}},w_{a})\ldots \vac \rangle = 
\prod_{a<b}
(w_{a}-w_{b})^{\alpha^{a}\cdot \alpha^{b}},
\end{align}
where the product is taken over $a,b\in\I$ with ordering $-1<1<\ldots -g<g$. 
\begin{proposition}
	\label{prop:ZMalpha}
	The  genus $g$ partition function for  $M^{2}_{\bm{\alpha}}$  is
	\[
	Z_{M^{2}_{\bm{\alpha}}}^{(g)}(\bm{w,\rho})
	=\frac{e^{\im\pi\bm{\alpha}.\Omega.\bm{\alpha}}}{\det(I-A)}, 
	\]
	where $
\bm{\alpha}.\Omega.\bm{\alpha}=\sum_{a,b\in\Ip}(\alpha^{a}\cdot\alpha^{b})\Omega_{ab}$
for  genus $g$ period matrix $\Omega$.
\end{proposition}
\begin{proof}
In this case we sum over  $b_{a}=h^{\alpha^{a}}(\bi_{a},\bi_{-a})$ of \eqref{eq:Fockalpha} labelled by $2g$ multisets with dual Fock basis vectors of \eqref{eq:Fockdualalpha}. We find
	\begin{align*}
	\bm{\rho^{\wt(b^{\alpha})}}Z^{(0)}(\bm{b^{\alpha},w})=&
\prod_{a\in\Ip}(-\rho_{a})^{\half\alpha^{a}\cdot\alpha^{a}} \frac{M(\bm{\rho},\bi_{a})}{\brfact}
	Z^{(0)}\left(\ldots ;h^{\alpha^{a}}(\bi_{a},\bi_{-a}),w_{a};\ldots \right).
	\end{align*}
$Z^{(0)}\left(\ldots ;h^{\alpha^{a}}(\bi_{a},\bi_{-a}),w_{a};\ldots \right)$ is determined by the coefficient of an appropriate expansion of the generating function \eqref{eq:2nptalpha} with 
$\theta^{\pm}(x )= \sum_{b\in\Ip}\alpha_{\pm}^{b}(\frac{1}{x -w_{b}}-\frac{1}{x-w_{-b}})$. 
We find that 
	\begin{align*}
	Z^{(0)}\left(\ldots ;h^{\alpha^{a}}(\bi_{a},\bi_{-a}),w_{a};\ldots \right)
	=\prod_{a<b}	(w_{a}-w_{b})^{\alpha^{a}\cdot \alpha^{b}}
	\pperm_{\,\widehat{\phi}^{-} ,\widehat{\phi}^{+}} \Ah\left(\bi_{a},\bi_{a}\right),
	\end{align*}	
for $a,b\in\I$ as in  \eqref{eq:Z0lwalpha}, $\widehat{A}$ of \eqref{eq:Ahat} and
	\begin{align}
	\label{eq:phihat}
\widehat{\phi}^{\pm}_{a}(k)=&\coeff_{x^{k-1}}\theta^{\pm}(x+w_{\mp a})=\sqrt{k}\rho_{a}^{-k/2}\sum_{b\in\Ip}\alpha_{\pm}^{b}d_{b}^{\mp a}(k),\quad (k\in\bi_{a})
	\end{align}
for $d^{a}_{b}(k)$ of \eqref{eq:ddef}. 
	The  $-\rho_{a}^{i_{a}}/{i_{a}}$ multiplicative factors in $M(\bm{\rho},\bi_{a})$ are absorbed into the $\Ah$ and $\widehat{\phi}^{\pm}$ terms using Theorem ~\ref{theor:MMTpperm}
	and Proposition~\ref{prop:ZM}   to obtain
	\begin{align*}
		Z_{M^{2}_{\bm{\alpha}}}^{(g)}(\bm{w,\rho})
	=&
F(\bm{w,\rho})
	\sum_{\bi_{a}}\frac{\pperm_{\,\phi^{-},\phi^{+}} A\left(\bi_{a},\bi_{a}\right)}{\brfact}
=
F(\bm{w,\rho})
		\frac{e^{\phi^{-}(I-A)^{-1}(\phi^{+})^{T}}}{\det(I-A) },
\end{align*}
where 
\begin{align}
\phi^{\pm}_{a}(k):=&\mp\frac{\rho_{a}^{k/2}}{\sqrt{k}}\widehat{\phi}_{a}^{\pm}(k)=\mp\sum_{b\in\Ip}\alpha_{\pm}^{b}d_{b}^{\mp a}(k),\quad (k\in\bi_{a})
\label{eq:phika}
\\
\label{eq:Fwrho}
F(\bm{w,\rho}):
	=&	\prod_{a\in\Ip}
	\left(\frac{ -\rho_{a}}{(w_{a}-w_{-a})^{-2}}\right)^{\half\alpha^{a}\cdot\alpha^{a}}
	\prod_{a,b\in\Ip;a\neq b} 	
	\left(\frac{(w_{a}-w_{b})(w_{-a}-w_{-b})}{(w_{-a}-w_{b})(w_{-a}-w_{-b})}\right)^{\half \alpha^{a}\cdot \alpha^{b}},
	\end{align}
	recalling that $\alpha^{-a}=-\alpha^{a}$. Noting that 
	$\phi^{-}(I-A)^{-1}\phi^{+}=-\sum_{a,b\in\Ip}\alpha^{a}_{-}\alpha^{b}_{+}
	d_{a}(I-A)^{-1}\overline{d}_{b}$ the result follows on comparison with the expressions for $\tpi \Omega_{ab}$ in \eqref{eq:Omgab} and \eqref{eq:Omgaa}.
\end{proof}

\subsection{Genus $g$ correlation functions on $\M^{2}$}
We define genus $g$   $n$-point correlation differential forms on 
$M^{2}_{\bm{\alpha}}=\otimes_{a\in\Ip}M^{2}_{\alpha^{a}}$ for $n$ vectors 
 $v_{i}\otimes e^{\beta^{i}}\in \M^{2}$ inserted at $x_{i}$  by
\begin{align}\label{eq:Fnptalpha}
\F_{M^{2}_{\bm{\alpha}}}^{(g)}(\bm{v\otimes e^{\beta},x}):=
\sum_{\bm{b}\in M^{2}_{\bm{\alpha}}}\bm{\rho^{\wt(b)}}Z^{(0)}(\bm{v\otimes e^{\beta},x};\bm{b,w})\bm{dx^{\wt(v\otimes e^{\beta})}},
\end{align}
for  $
Z^{(0)}(\bm{v,x};\bm{b,w})=
Z^{(0)}(\ldots;v_{i}\otimes e^{\beta^{i}},x_{i};\ldots;\bbar_{a},w_{-a};b_{a},w_{a};\ldots)$.

We describe the genus $g$ generating function for all correlation functions \eqref{eq:Fnptalpha}  generalizing Proposition~\ref{prop:ZMnptg} and Corollary~\ref{cor:Zohalpha}. The result is expressed in terms of 
the genus $g$ bidifferential $\omega(x,y)$, the normalized 1-form $\nu_{a}(x)$,  the period matrix $\Omega$, the differential of the third kind $\omega_{p-q}(x)$ and the prime form $E(x,y)$ of \eqref{eq:omegag}--\eqref{eq:prime} as follows:
\begin{theorem}	\label{theor:Zmain}  
For Heisenberg generators $h_{\pm}$ inserted at $y_{r}^{\pm}$ for $r\in \{1,\ldots,m\}$ and  $e^{\beta^{t}}$, for $\beta^{t}\in\C^2$ with $\sum_{t=1}^{n}\beta^{t}= 0$, inserted at $z_{t}$ for $t\in \{1,\ldots,n\}$ we have
\begin{align}
\label{eq:Zmain}
\F_{M^{2}_{\bm{\alpha}}}^{(g)}(\bm{h_{\pm},y^{\pm};e^{\beta},z })=&
\varepsilon_{\bm{\beta}} 
\prod_{t<u} E(z_{t},z_{u})^{\beta^{t}\cdot \beta^{u}}
\pperm_{\,\widetilde{\theta}^{-} , \widetilde{\theta}^{+} }\omega\left(y^{+}_{r},y^{-}_{s}\right) 
\\
\notag
& \times	
\exp\left(\im\pi\bm{\alpha}.\Omega.\bm{\alpha}
	+
	\sum_{a, t}
	\alpha^{a}\cdot \beta^{t}\,\int^{z_{t}}_{z_{0}}\nu_{a}
	\right) \det(I-A)^{-1},
\end{align}
for $a\in\Ip$; $t,u\in \{1,\ldots,n\}$ with 
\begin{align}
\label{eq:thetatilde}
\widetilde{\theta}^{\pm}_{r}:= \sum_{a\in\Ip}\alpha^{a}_{\pm}\nu_{a}(y_{r}^{\pm})
+\sum_{t=1}^{n}\beta^{t}_{\pm}\omega_{z_{t}-z_{0}}(y_{r}^{\pm}),\quad r\in \{1,\ldots,m\},
\end{align}
for arbitrary choice of $z_{0}$ in \eqref{eq:Zmain} or  \eqref{eq:thetatilde}. 
\end{theorem}
\begin{proof}
The proof is similar to that of Propositions~\ref{prop:ZMnptg} and \ref{prop:ZMalpha}. Here the sum in  \eqref{eq:Fnptalpha}  is taken over $b_{a}=h_{\alpha_{a}}(\bi_{a},\bi_{-a})$ with summand 
	\begin{align*}
	&\bm{\rho^{\wt(b)}}Z^{(0)}(\bm{\bm{h_{\pm},y^{\pm}};b,w})=
	\prod_{a\in\Ip}(-\rho_{a})^{\half\alpha^{a}\cdot\alpha^{a}} 
	\frac{M(\bm{\rho},\bi_{a})}{\brfact}X_{\bm{\alpha}}(\bi_{a}),
	\end{align*}
	for $X_{\bm{\alpha}}(\bi_{a})=Z^{(0)}\left(\ldots;h_{+},y^{+}_{r};\ldots;h_{-},y^{-}_{s} \ldots;e^{\beta^{t}},z_{t}; \ldots ;h_{\alpha^{a}}
	(\bi_{a},\bi_{-a}),w_{a};\ldots \right)$	
	determined by expansions of \eqref{eq:2nptalpha} with 
	\begin{align*}
	\theta^{\pm}(x )= \sum_{b\in\Ip}\alpha_{\pm}^{b}\left(\frac{1}{x -w_{b}}-\frac{1}{x-w_{-b}}\right)
	+\sum_{t=1}^{n}\beta_{\pm}^{t}\left(\frac{1}{x-z_{t}}-\frac{1}{x-z_{0}}\right).
\end{align*} 
for arbitrary choice of $z_{0}$ (since $\sum_{t=1}^{n}\beta^{t}= 0$).
	We find that 
	\begin{align*}
	X_{\bm{\alpha}}(\bi_{a})=&\varepsilon_{\bm{\beta}}
	\prod_{a,s}(z_{s}-w_{a})^{\alpha^{a}\cdot \beta^{s}}
	\prod_{ t<u}	(z_{t}-z_{u})^{\beta^{t}\cdot \beta^{u}}
	\prod_{a<b}	(w_{a}-w_{b})^{\alpha^{a}\cdot \alpha^{b}}
	\\
	&\times 
	\pperm_{\,\widehat{\Theta}^{-},\widehat{\Theta}^{+}}
	\begin{bmatrix}
	\Bh&  \Uh(\bi_{a})\\ 
	\Vh(\bi_{a}) & \Ah\left(\bi_{a},\bi_{a}\right)
	\end{bmatrix},
	\end{align*}
	for product indices $a,b\in\I$; $s,t,u\in \{1,\ldots,n\}$,  with $\Ah$ of \eqref{eq:Ahat} and $\Bh, \Uh,\Vh$ of \eqref{eq:BUVhat} and where 
	$\widehat{\Theta}^{\pm}=(\ldots, \theta^{\pm}(y_{r}^{\pm}),\ldots \widehat{\psi}^{\pm}_{a}(k),\ldots)$ for $r\in\{1,\ldots,m\} $ and $ k\in\bi_{a}$ with
	\begin{align*}
	\widehat{\psi}^{\pm}_{a}(k)=&
	\coeff_{x^{k-1}}\theta^{\pm}(x+w_{\mp a})=
	\widehat{\phi}^{\pm}_{a}(k)+\sum_{t=1}^{n}\beta_{\pm}^{t}
	\left((z_{0}-w_{\mp a})^{-k-1}-(z_{t}-w_{\mp a})^{-k-1}\right),
	\end{align*}
for $\widehat{\phi}^{\pm}$ of \eqref{eq:phihat}.
Absorbing 
$-\rho_{a}^{i_{a}}/{i_{a}}$ and $\bm{dy^{+}_{r}dy^{-}_{s}}$ factors into the partial permanent
we find
	\begin{align*}
\F_{M^{2}_{\bm{\alpha}}}^{(g)}(\bm{\bm{h_{\pm},y^{\pm};e^{\beta},z}})
&=G(\bm{w,\rho})\sum_{\bi_{a}}\frac{1}{\brfact}
\pperm_{\, \Theta^{-},\Theta^{+}}
\begin{bmatrix}
B&  U(\bi_{a})\\ 
V(\bi_{a}) & A\left(\bi_{a},\bi_{a}\right)
\end{bmatrix},
\end{align*}
for $B,U,V$ of \eqref{eq:BUV} where with  $F(\bm{w,\rho})$ of \eqref{eq:Fwrho} we define
\begin{align}
\label{eq:Gterm}
G(\bm{z,w,\rho}):=&\varepsilon_{\bm{\beta}}F(\bm{w,\rho})
\prod_{a,s}\left(\frac{z_{s}-w_{a}}{z_{s}-w_{-a}}\right)^{\alpha^{a}\cdot \beta^{s}}
\prod_{ t<u}	(z_{t}-z_{u})^{\beta^{t}\cdot \beta^{u}}\bm{dz^{\half \beta^2}},
\end{align}
for $a\in\Ip$; $s,t,u\in\{1,\ldots,n\}$. 
 $\Theta^{\pm}=(\ldots,\theta^{\pm}_{r},\ldots,\psi^{\pm}_{a}(k),\ldots) $  where
\begin{align*}
\theta^{\pm}_{r}:=& \theta^{\pm}(y^{\pm}_{r})dy^{\pm}_{r}=
\sum_{b\in\Ip}\alpha_{\pm}^{b}\om{0}{w_{b}-w_{-b}}(y^{\pm}_{r})
+\sum_{t=1}^{n}\beta_{\pm}^{t}\om{0}{z_{t}-z_{0}}(y^{\pm}_{r}),
\\
\psi_{a}^{\pm}(k):=&\mp \frac{\rho_{a}^{k/2}}{\sqrt{k}}\widehat{\psi}^{\pm}_{a}(k)
=\mp\left(   
\sum_{b\in\Ip}\alpha_{\pm}^{b}d_{b}^{\mp a}(k)
+\sum_{t=1}^{n}\beta_{\pm}^{t}\int_{z_{0}}^{z_{t}}L_{\mp a}(k,\cdot)
\right),
\end{align*}
and $L_{a}(k,x)$ of \eqref{eq:Lkdef}.
By the general McMahon Master Theorem~\ref{theor:MMTsubpperm} we find
\begin{align}
\label{eq:fmid}
\F_{M^{2}_{\bm{\alpha}}}^{(g)}(\bm{\bm{h_{\pm},y^{\pm};e^{\beta},z}})=
G(\bm{w,\rho})\,
\frac{e^{\psi^{-}(I-A)^{-1}(\psi^{+})^{T}}}{\det(I-A)}
\pperm_{\,\widetilde{\theta}^{-} , \widetilde{\theta}^{+} }\widetilde{B},
\end{align}
with $\widetilde{B}(r,s)=\omega(y^{+}_{r},y^{-}_{s})$ from \eqref{eq:Btilde} 
and $\widetilde{\theta}^{+}= \theta^{+}+U(I-A)^{-1}\psi^{+}$ given by
\begin{align*}
\widetilde{\theta}^{+}_{r}
=&\sum_{b\in\Ip}\alpha_{+}^{b}\left( \om{0}{w_{b}-w_{-b}}(y^{+}_{r})
-L(y_{r}^{+} )(I-A)^{-1}\overline{d}_{b}\right)
\\
&+\sum_{t=1}^{n}\beta_{+}^{t}\left( \om{0}{z_{t}-z_{0}}(y^{+}_{r})
-L(y_{r}^{+}) (I-A)^{-1}\int_{z_{0}}^{z_{t}}R(\cdot)\right)
\\
&=\sum_{a\in\Ip}\alpha^{a}_{+}\nu_{a}(y_{r}^{+})
+\sum_{t=1}^{n}\beta^{t}_{+}\omega_{z_{t}-z_{0}}(y_{r}^{+}),
\end{align*}
by Propositions~\ref{prop:omgsew} and \ref{prop:nu}. A similar result holds for $\widetilde{\theta}^{-}$ so that we obtain the $\pperm_{\,\widetilde{\theta}^{-} , \widetilde{\theta}^{+} }\omega\left(y^{+}_{r},y^{-}_{s}\right) $ term in \eqref{eq:Zmain}. The exponent $\psi^{-}(I-A)^{-1}(\psi^{+})^{T}$  term in \eqref{eq:fmid} is given by
\begin{align}
\label{eq:terms}
& 
-\sum_{a,b\in\Ip}\alpha_{-}^{a}\alpha_{+}^{b}d_{a}(I-A)^{-1}\overline{d}_{b}
-\sum_{a\in\Ip}\sum_{t=1}^{n}\alpha_{-}^{a}\beta_{+}^{t}\int_{z_{0}}^{z_{t}}
d_{a}(I-A)^{-1}R(\cdot)
\\
\notag
&-\sum_{a\in\Ip}\sum_{t=1}^{n}\beta_{-}^{t}\alpha_{+}^{a}\int_{z_{0}}^{z_{t}}
L(\cdot)(I-A)^{-1}\overline{d}_{a}
-\sum_{t,u=1}^{n}\beta_{-}^{t}\beta_{+}^{u}\int_{x_{0}}^{z_{t}}L(\cdot)(I-A)^{-1}\int_{y_{0}}^{z_{u}}R(\cdot),
\end{align}
for arbitrary $x_{0},y_{0},z_{0}$. The $\alpha_{-}^{a}\alpha_{+}^{b}$ terms are combined with the 
$F(\bm{w,\rho})$ term to obtain the $\im\pi\bm{\alpha}.\Omega.\bm{\alpha}$ exponent in \eqref{eq:Zmain} as in the proof of Proposition~\ref{prop:ZMalpha}. 
Proposition~\ref{prop:nu} implies 
\begin{align*}
\int_{z_{0}}^{z_{t}}\nu_{a}
-\log\left(\frac{z_{t}-w_{a}}{z_{t}-w_{-a}}\right)
&=
-\int_{z_{0}}^{z_{t}}
d_{a}(I-A)^{-1}R(\cdot)
=
-\int_{z_{0}}^{z_{t}}
L(\cdot)(I-A)^{-1}\overline{d}_{a}.
\end{align*}
Thus the $\alpha^{a}\cdot\beta^{t}$ terms in \eqref{eq:Gterm} and  \eqref{eq:terms} combine to give the  $\displaystyle{\sum_{a, t}
\alpha^{a}\cdot \beta^{t}\int^{z_{t}}_{z_{0}}\hspace{-2 mm}\nu_{a}}$ term  in \eqref{eq:Zmain}. 

Let  $r(x,y):= \log \left(\frac{K(x,y)}{ x-y}\right)$ 
for $E(x,y)=K(x,y)dx^{-\half}dy^{-\half}$ of   \eqref{eq:prime} so that
$r(x,y)=r(y,x)$ and  $r(x,x)=0$. 
From Proposition~\ref{prop:omgsew} and since $\sum_{t=1}^{n}\beta^{t}= 0$ we find
\begin{align*}
-\sum_{t,u=1}^{n}\beta_{-}^{t}\beta_{+}^{u}\int_{x_{0}}^{z_{t}}L(\cdot)(I-A)^{-1}\int_{y_{0}}^{z_{u}}R(\cdot)
=&
\sum_{t,u=1}^{n}\beta_{-}^{t}\beta_{+}^{u}\int_{x_{0}}^{z_{t}}\int_{y_{0}}^{z_{u}}
\left(\omega(x,y)-\omega^{(0)}(x,y)\right)
\\
=&
\sum_{t,u=1}^{n}\beta_{-}^{t}\beta_{+}^{u} r(z_{t},z_{u})=\sum_{ t<u}\beta^{t}\cdot \beta^{u} r(z_{t},z_{u}).
\end{align*}
We  combine these $\beta^{t}\cdot \beta^{u} $ terms  in \eqref{eq:terms} with those  in $G(\bm{z,w,\rho})$ together with the differential $\bm{dz^{\half \beta^2}}$ to obtain the $E(z_{t},z_{u})^{\beta^{t}\cdot \beta^{u}}$ terms in \eqref{eq:Zmain}. Thus the theorem is proved.
\end{proof}
\begin{remark}
	$\F_{M^{2}_{\bm{\alpha}}}^{(g)}(\bm{h,y;e^{\beta},z })$ is the generating function for all correlation  functions on $\M^2$ for the Fock vectors \eqref{eq:Fockalpha} by repeated use of \eqref{eq:genrel}. Thus for the Virasoro vector, 
	\[
	\F_{M^{2}_{\bm{\alpha}}}^{(g)}(\omega,z)=\coeff_{x^{0}}\F_{M^{2}_{\bm{\alpha}}}^{(g)}(h_{+},x+z;h_{-},z)=\frac{1}{6}s(z)Z_{M^{2}_{\bm{\alpha}}}^{(g)},
	\]
	 for the projective connection 
	$s(x):=6 \lim_{y\rightarrow x}\left( \omega(x,y)-\frac{1}{(x-y)^2}dxdy\right)$.

\end{remark}

Choosing $\alpha^{a}=(\alpha^{a}_{1},0)$, $\beta^{t}=(\beta^{t}_{1},0)$ in Theorem~\ref{theor:Zmain} we obtain the generating function 
for all correlation  functions on $M_{\bm{\alpha}}=\otimes_{a\in\Ip}M_{\alpha^{a}}$ for $\alpha^{a}\in\C$ (on relabelling) as follows:
\begin{corollary}	\label{cor:Zmain}  
	For Heisenberg generators $h$ inserted at $y_{r}$ for $r\in \{1,\ldots,m\}$ and  $e^{\beta^{t}}$ for $\beta^{t}\in\C$ with $\sum_{t=1}^{n}\beta^{t}= 0$ inserted at $z_{t}$ for $t\in \{1,\ldots,n\}$ we have
	\begin{align}
	\label{eq:Zmaincor}
	\F_{M_{\bm{\alpha}}}^{(g)}(\bm{h,y;e^{\beta},z })=&
	\varepsilon_{\bm{\beta}} 
	\prod_{t<u} E(z_{t},z_{u})^{\beta^{t}\beta^{u}}
	\mathrm{Sym}_{m}\left(\omega ,\nu_{\bm{\alpha,\beta}}\right)	
	\frac{\exp\left(\im\pi\bm{\alpha}.\Omega.\bm{\alpha}+
		\displaystyle{\sum_{a, t}
		\alpha^{a} \beta^{t}\int^{z_{t}}_{z_{0}}\nu_{a}}
		\right)}{\det(I-A)^{\half}},
	\end{align}
for $a\in\Ip$; $t,u\in \{1,\ldots,n\}$ with $\bm{\alpha}.\Omega.\bm{\alpha}=\sum_{a,b\in\Ip}\alpha^{a}\Omega_{ab}\alpha^{b}$  and where  
	\begin{align*}
\nu_{\bm{\alpha,\beta}}:=\sum_{a\in\Ip}\alpha^{a}\nu_{a}+\sum_{t=1}^{n}\beta^{t}\omega_{z_{t}-z_{0}},
	\end{align*}
for arbitrary choice of $z_{0}$	and symmetric
	(tensor) product of $\nu_{\bm{\alpha,\beta}}$ and $\omega$ defined by
	\begin{align*}
	\mathrm{Sym}_m\big(\omega ,\nu_{\bm{\alpha,\beta}}\big):=
	\sum_{\varphi}\prod_{q}\nu_{\bm{\alpha,\beta}}(y_{q})\prod_{(r,s)}\omega (y_r,y_s),
	\end{align*}
	where we sum over all inequivalent  involutions $\varphi= (q)\ldots  (rs)\ldots$ of the labels $\{1,\ldots,m\}$. 
\end{corollary}

\begin{example}
	Consider the lattice (Super)VOA $V_{L}$ for a  rank $d$ even (odd) integral lattice $L$.
	Then $V_{L}=\oplus_{\lambda\in L}M^{d}_{\lambda}$ where 
	$ M^{d}_{\lambda}=\otimes_{i=1}^{d}M_{\lambda_{i}}$ and $\lambda=(\lambda_{1},\ldots ,\lambda_{d})\in L$ so that
	\begin{align*}
	Z_{V_{L}}^{(g)}=\frac{\Theta_{L}(\Omega)}{\det(I-A)^{d/2}},
	\end{align*}
	for genus $g$ Siegel lattice theta function  $\Theta_{L}(\Omega)=\sum_{\bm{\lambda}\in \otimes^{g}L}e^{\im\pi\,\bm{\lambda}.\Omega.\bm{\lambda}}$.
	For even $L$, this agrees with Proposition~6.9 of  \cite{TW}  found  by the alternative technique of genus $g$ Zhu recursion.
\end{example}
\begin{example}
For  $\alpha^{a}=\beta^{t}=0$  the correlation  function $\F_{M}^{(g)}(\bm{h,y})$ agrees with Proposition~6.1 of  \cite{TW} found  from  genus $g$ Zhu recursion. 
\end{example}
\begin{example}
Consider the correlation function for $e^{+1}$ inserted at $x_{i}$ and $e^{-1}$ inserted at $y_{j}$ for $i,j\in \{1,\ldots,n\}$ on $M_{\bm{m+\alpha}}=\otimes_{a\in\Ip}M_{m^{a}+\alpha^{a}}$ for
$\bm{m}\in\Z^{g}$  and  $\bm{\alpha}\in\R^{g}$. Then 
	\begin{align*}
\F_{M_{\bm{m+\alpha}}}^{(g)}(\bm{e^{+1},x ;e^{-1},y })=&
\frac{\prod_{i<j} E(x_{i},x_{j})E(y_{i},y_{j})}{\prod_{i,j} E(x_{i},y_{j})}
\frac{
e^{\im\pi(\bm{m+\alpha}).\Omega.(\bm{m+\alpha})+(\bm{m+\alpha}).\bm{\zeta}}
}
{\det(I-A)^{\half}},
\end{align*}
where $\zeta_{a}=\displaystyle{\sum_{i=1}^{n}\int_{y_{i}}^{x_{i}}\nu_{a}}$. Summing over all $\bm{m}\in\Z^{g}$ we find
\begin{align}
\label{eq:Zfermion}
\sum_{\bm{m}\in\Z^{g}}\F_{M_{\bm{m+\alpha}}}^{(g)}(\bm{e^{+1},x ;e^{-1},y })
=
\frac{\prod_{i<j} E(x_{i},x_{j})E(y_{i},y_{j})}{\prod_{i,j} E(x_{i},y_{j})}
\frac{\Theta[\bm{\alpha}](\Omega,\bm{\zeta})}
{\det(I-A)^{\half}},
\end{align}
for Riemann theta function $\Theta[\bm{\alpha}](\Omega,\bm{\zeta}):=	\sum_{\bm{m}\in\Z^{g}}
e^{\im\pi(\bm{m+\alpha}).\Omega.(\bm{m+\alpha})+(\bm{m+\alpha}).\bm{\zeta}}$. 
\eqref{eq:Zfermion} is the correlation generating function for twisted sectors of the rank 2 fermionic SVOA e.g. \cite{R}. 
\end{example}

\section{Appendix}\label{sect:append}
We  describe some  elementary properties of $\om{0}{}(x,y)$ for $x,y\in \Chat$. 
For $|x|>|y|$ we find
\begin{align}
\om{0}{}(x,y)=-b(x) c (y),
\label{eq:om0bc}
\end{align}
where $b(x)$, $ c (y)$ are infinite row and column vectors with components given by the 1-forms
\begin{align*}
b_{l}(x)
:= &
\frac{1}{\tpi\sqrt{l}}
\oint\limits_{\Con{0}(y)}
y^{-l}\om{0}{}(x,y)
=\sqrt{l}x^{-l-1}\,dx,\quad l=1,2,\ldots,
\notag\\
 c _{k}(y)
:= &
-\frac{1}{\tpi\sqrt{k}}
\oint\limits_{\Con{\infty}(x)}
x^{k}\om{0}{}(x,y)
=-\sqrt{k}y^{k-1}\,dy,\quad k=1,2,\ldots,
\end{align*}
for Jordan curves $\Con{0} $  in the neighborhood of $0$ and $\Con{\infty}  $ 
 in the neighborhood of $\infty$.
Note that $\Con{\infty}(x)\sim -\Con{0}(x^{-1})$ implies 
$b_k(x)= c _{k}(x^{-1})$.

For  $\gamma\in \SL_{2}(\C)$ define a 
matrix  $D(\gamma)$ with components for $k,l\ge 1$ given by  \cite{DV}
\begin{align}
D_{kl}(\gamma)
:= &
\frac{1}{(\tpi)^2\sqrt{kl}}
\oint\limits_{\Con{\infty}(x)}
\oint\limits_{\Con{0}(y)}
x^{k}y^{-l}\om{0}{}(x,\gamma y)=
\begin{cases}
\sqrt{\frac{l}{k}}\partial_{y}^{(l)}(\gamma y)^k
\Big|_{y=0}, & \gamma(0)\neq \infty,
\\
0, & \gamma(0)=\infty,
\end{cases}
\label{eq:Dkl}
\end{align}
where $\partial_{y}^{(l)}:=\frac{1}{l!}\frac{\partial^{l}}{\partial y^{l} }$.
\begin{lemma} 
	\label{lem:Drep}
Let $\gamma_1,\gamma_2\in \SL_{2}(\C)$ such that $\gamma_{1}(0), \gamma_{2}(0), \gamma_{1}\gamma_{2}(0)\neq \infty$. Then
\begin{enumerate}
	\item[(i)]$b(x)D(\gamma_1)= b(\gamma_1^{-1} x)$,
	\item[(ii)]$D(\gamma_1) c (y)=  c (\gamma_1 y)$,
	\item[(iii)]$D(\gamma_1)D(\gamma_2)= D(\gamma_1\gamma_2)$.
\end{enumerate}
\end{lemma}

\begin{proof} $\SL_{2}(\C)$ invariance of $\om{0}{}(x,y)$ implies
	\[
	\om{0}{}(\gamma_{1}^{-1}x, y)
	=
	\om{0}{}(x,\gamma_{1} y)
	=-b(x)  c (\gamma_{1} y),
	\] 
	for $|x|>|\gamma_{1} y|$ from  \eqref{eq:om0bc}. 
	Provided $\gamma_{1}(0)\neq \infty$ we may compute $y$ moments of both sides of this equation to obtain  (i). A similar argument leads to (ii).  
(iii) follows by considering $y$ moments of (ii). 
\end{proof}
\begin{remark}
	Note that (iii) does \textbf{not} imply that $D$  is a representation of $\SL_{2}(\C)$ e.g. for $\gamma=\bigl(\begin{smallmatrix}
	0 & 1\\ 1 & 0
	\end{smallmatrix} \bigr)$ then 
	$I=D(\gamma^2)\neq D(\gamma)^2=0$. 
\end{remark}

We may now readily express the 1-forms $L(x)$, $R(x)$ and the moment matrix $A$   of \eqref{eq:Lkdef}, \eqref{eq:Adef}. Define $\lambda_{a},\mu_{a}\in \SL_{2}(\C)$ for $a\in \I$ by
\begin{align}
\lambda_{a}:= 
\begin{pmatrix}
\rho_{a}^{-1/2} & 0\\
0 & 1
\end{pmatrix}
\begin{pmatrix}
1 & -w_{a}\\
0 & 1
\end{pmatrix},
\quad
\mu_{a}:= 
\begin{pmatrix}
\rho_{a}^{1/2} & 0\\
0 & 1
\end{pmatrix}
\begin{pmatrix}
0 & 1\\
1 & -w_{-a}
\end{pmatrix}.
\label{eq:lammu}
\end{align}
The sewing condition \eqref{eq:sew} reads $\lambda_{a}z'=1/(\lambda_{-a}z)=\mu_{a}z$ i.e. the Schottky generators obey
\begin{align}
\gamma_{a}=\lambda_{a}^{-1}\mu_{a}. 
\label{eq:gammaSchot}
\end{align}
It is easy to check that the 1-forms $L(x)$, $R(y)$ and the moment matrix $A$  are given by
\begin{align}
L_{a}(k,x)=b_{k}(\lambda_{a}x),
\quad 
R_{a}(l,y)=  c _{l}(\mu_{a}y), 
\quad
A_{ab}(k,l)= D_{kl}(\mu_{a}\lambda_{b}^{-1}).
\label{eq:aA}
\end{align}
Lemma  \ref{lem:Drep}, \eqref{eq:om0bc} and \eqref{eq:gammaSchot} imply that for integer $n\ge 1$
\begin{align*}
L(x)A^{n-1}R(y)= &
\sum_{a_{1},\ldots,a_{n}}
b(\lambda_{a_{1}}x)D(\mu_{a_{1}}\lambda_{a_{2}}^{-1})\ldots D(\mu_{a_{n-1}}\lambda_{a_{n}}^{-1}) c (\mu_{a_{n}}y)
\notag
\\
= &
\sum_{a_{1},\ldots,a_{n}}
b(x) c (\gamma_{a_{1}}\ldots \gamma_{a_{n}}y)
= -\sum_{\gamma \in \Gamma _{n}}\om{0}{}(x,\gamma y),
\end{align*}
where $\Gamma_{n}$ denotes the elements of the Schottky group consisting of reduced words of length $n$ i.e. for all adjacent generators $\gamma_{-a_{i}}\equiv\gamma_{a_{i}}^{-1}\neq \gamma_{a_{i+1}}$ for $i=1,\ldots,n-1$.
Hence Proposition~\ref{prop:omgsew} implies the classical  formula e.g. \cite{Bo}
	\begin{align}
	\omega (x,y)=\sum_{\gamma \in \Gamma }\om{0}{}(x,\gamma y),
	\label{eq:omgSch}
	\end{align}
	where the sum is taken over all the elements of the Schottky group $\Gamma $.


\begin{thebibliography}{McIT}
\bibitem[BK]{BK} Bakalov, B. and Kac, V.:
Generalized vertex algebras, math.QA/0602072.	
		
\bibitem[Bo]{Bo} Bobenko, A.,
Introduction to compact Riemann surfaces,
in  \textit{Computational Approach to Riemann Surfaces}, 
edited Bobenko, A. and Klein, C.,  Springer-Verlag (Berlin, Heidelberg, 2011).
	
\bibitem[C]{C} Cohen, D.E.:
\textit{Combinatorial group theory: a topological approach}, 
LMS Student Texts \textbf{14}, Cambridge University Press, (1989).

\bibitem[DL]{DL} Dong, C and Lepowsky, J.: 
\textit{Generalized Vertex Algebras and Relative Vertex Operators,}
Progress in Mathematics, \textbf{112}, Birkhauser (Boston,  1993).

\bibitem[DP]{DP} D'Hoker, E.  and Phong, D.H.: 
On determinants of Laplacians on Riemann surfaces, 
Comm.Math.Phys. \textbf{104} (1986) 537--545.

\bibitem[DV]{DV} Di Vecchia, P.,  Pezzella, F.,  Frau, M.,  Hornfeck, K.,  Lerda, A. and Sciuto, S.: 
N-point g-loop vertex for a free bosonic theory with vacuum charge Q,  
Nucl.Phys. \textbf{B322}  (1990)  317--372. 

\bibitem[F]{F} Fay, J.D.: 
\textit{Theta Functions on Riemann Surfaces},
Lecture Notes in Mathematics, Springer-Verlag, (Berlin-New York, 1973). 

\bibitem[FHL]{FHL} Frenkel, I., Huang, Y. and Lepowsky, J.: 
On axiomatic approaches to vertex operator algebras and modules, 
Mem.Amer.Math.Soc. \textbf{104} (1993).

\bibitem[FLM]{FLM} Frenkel, I., Lepowsky, J. and Meurman, A.: 
\textit{Vertex Operator Algebras and the Monster}, 
Academic Press, (New York, 1988).

\bibitem[Fo]{Fo} Ford, L.R.:
\textit{Automorphic Functions},
AMS-Chelsea, (Providence, 2004).

\bibitem[Gu]{Gu} Gunning, R.C.: 
\textit{Introduction to Holomorphic Functions of Several Variables Vol 1}, 
Wadsworth \& Brooks/Cole, (Belmont, 1990).

\bibitem[HT] {HT} Hurley, D. and Tuite, M.P.:
Virasoro correlation functions, 
Int.J.Math. \textbf{23} (2012) 1250106.

\bibitem[Ka]{Ka} Kac, V.: 
\textit{Vertex Operator Algebras for Beginners},
University Lecture Series \textbf{10} AMS, (1998).

\bibitem[Li]{Li} Li, H.: 
Symmetric invariant bilinear forms on vertex operator algebras, 
J.Pure.Appl.Alg. \textbf{96} (1994)  279--297.

\bibitem[LL]{LL} Lepowsky, J.  and Li, H.:
\textit{Introduction to Vertex Operator Algebras and Their Representations}, 
Progress in Mathematics \textbf{227}, Birkh\"auser, (Boston, 2004).

\bibitem[McIT]{McIT} McIntyre, A. and Takhtajan, L.A.: 
Holomorphic factorization of determinants of Laplacians on Riemann surfaces and higher genus generalization of Kronecker's first limit formula, 
GAFA, Geom.Funct.Anal. \textbf{16} (2006) 1291--1323.

\bibitem[McM]{McM} MacMahon, P.A.:  
\textit{Combinatory Analysis}, Vol. 1, 
Cambridge University Press, (Cambridge 1915); reprinted by Chelsea, (New York, 1955).

\bibitem[Mo]{Mo} Montonen, C.:
Multiloop amplitudes in additive dual-resonance models,
Nuovo Cim. \textbf{19} (1974) 69--89.

\bibitem[MT1]{MT1} Mason, G. and Tuite, M.P.: 
On genus two Riemann surfaces formed from sewn tori, 
Commun.Math.Phys. \textbf{270} (2007)  587--634.

\bibitem[MT2]{MT2} Mason, G. and Tuite, M.P.: 
Torus chiral n-point functions for free boson and lattice vertex operator algebras, 
Commun.Math.Phys. \textbf{235} (2003) 47--68.

\bibitem[MT3]{MT3}  Mason, G. and Tuite, M.P.: 
Vertex operators and modular forms, \textit{A Window into Zeta and Modular Physics} 
eds. K.~Kirsten and F.~Williams, 
MSRI Publications \textbf{57} 183--278, Cambridge University Press, (Cambridge, 2010).

\bibitem[MT4]{MT4}  Mason, G. and Tuite, M.P.: 
Free bosonic vertex operator algebras on genus two Riemann surfaces~I, 
Commun.Math.Phys. \textbf{300} (2010) 673--713.

\bibitem[MT5]{MT5}  Mason, G. and Tuite, M.P.:
Free bosonic vertex operator algebras on genus two Riemann surfaces~II, 
in  \textit{Conformal Field Theory, Automorphic Forms and Related Topics}, 
Contributions in Mathematical and Computational Sciences \textbf{8} 183--225, Springer Verlag, (Berlin, Heidelberg, 2014). 

\bibitem[Mu]{Mu} Mumford, D.: 
\textit{Tata lectures on Theta I and II},
Birkh\"{a}user, (Boston, 1983).

\bibitem[R]{R} Raina, A.K.: 
Fay’s trisecant identity and conformal field theory,
Commun.Math.Phys. \textbf{122} (1989) 625--641. 

\bibitem[T]{T} Tuite, M.P.:
Some generalizations of the MacMahon master theorem, 
J.Comb.Theor.Ser. A \textbf{120} (2013) 92--101.

\bibitem[TW]{TW} Tuite, M.P. and Welby, M.: 
General genus Zhu recursion for vertex operator algebras, 
arXiv:1911.06596.

\bibitem[TZ]{TZ} Tuite, M.P. and Zuevsky, A.: 
A generalized vertex operator algebra for Heisenberg intertwiners,
J.Pure.Appl.Alg. \textbf{216} (2012) 1442--1453.

\bibitem[Y]{Y} Yamada, A.: 
Precise variational formulas for abelian differentials. 
Kodai Math.J. \textbf{3} (1980) 114--143.

\bibitem[Z]{Z} Zograf, P.G.: 
Liouville action on moduli spaces and uniformization of degenerate Riemann surfaces, 
(Russian)  Algebra i Analiz  \textbf{1} (1989) 136--160;  translation in  Leningrad Math. J. \textbf{1}  (1990) 941--965. 

\end{thebibliography}
\end{document}